\documentclass[11pt,reqno]{amsart}
\usepackage{indentfirst,enumerate,cite,amssymb,amsfonts,amsmath,amsthm,mathrsfs,dsfont}
\usepackage{color}
\usepackage{graphicx}
\usepackage{multicol}
\usepackage[colorlinks=true,linkcolor=blue,citecolor=blue]{hyperref}
\usepackage{enumerate}
\usepackage{geometry}
\numberwithin{equation}{section}

\newtheorem{theorem}{Theorem}[section]
\newtheorem{definition}[theorem]{Definition}

\newtheorem{lemma}[theorem]{Lemma}
\newtheorem{proposition}[theorem]{Proposition}
\newtheorem{corollary}[theorem]{Corollary}

\newtheorem{remark}{Remark}

\allowdisplaybreaks

\begin{document}

\title[Global Gevrey Hypoellipticity for Involutive Systems]{Global Gevrey Hypoellipticity of Involutive Systems on Non-Compact Manifolds}

	\author[S. Coriasco]{Sandro Coriasco}
	\address{
		Dipartimento di Matematica ``Giuseppe Peano'', 
		Università Degli Studi di Torino, 
		Via Carlo Alberto 10, CAP 0123, Torino, 
		Italia }
	\email{sandro.coriasco@unito.it}

	\author[A. Kirilov]{Alexandre Kirilov}
	\address{
		Departamento de Matem\'atica, 
		Universidade Federal do Paran\'a,  
		Caixa Postal 19096, CEP 81530-090, Curitiba, Paran\'a, 
		Brasil}
	\email{akirilov@ufpr.br}
	
	\author[W. de Moraes]{Wagner A. A. de Moraes}
	\address{
		Departamento de Matem\'atica,  
		Universidade Federal do Paran\'a,  
		Caixa Postal 19096\\  CEP 81530-090, Curitiba, Paran\'a, 
		Brasil}
	\email{wagnermoraes@ufpr.br}
	
	\author[P. Tokoro]{Pedro M. Tokoro}
	\address{
		Programa de P\'os-Gradua\c c\~ao em Matem\'atica,  
		Universidade Federal do Paran\'a,  
		Caixa Postal 19096\\  CEP 81530-090, Curitiba, Paran\'a,  
		Brasil}
	\email{pedro.tokoro@ufpr.br}

	

\begin{abstract}
	We investigate the global Gevrey hypoellipticity of a class of first-order differential operators associated with tube-type involutive structures on $M\times\mathbb{T}^m$, where $M$ is a non-compact manifold diffeomorphic to the interior of a compact manifold with boundary and $\mathbb{T}^m$ is the $m$-dimensional torus. For $s>1$, we work in Gevrey classes of Roumieu and Beurling type. A key step is the construction, on $M$, of a scattering metric whose coefficients are Gevrey of order $s$ in every analytic chart; this allows us to use Hodge theory and obtain Gevrey regularity for the harmonic forms. Under a natural condition on the defining closed $1$-forms, we obtain a sharp criterion for global Gevrey hypoellipticity in terms of rationality and (Roumieu/Beurling) exponential Liouville behavior.
	\end{abstract}

	\maketitle

	\section{Introduction}
	
	In this paper we provide a characterization of global hypoellipticity in Gevrey classes of Roumieu and Beurling type for a family of real involutive systems on a class of non-compact manifolds. 
	Let $M$ be an analytic paracompact manifold. For $s\ge1$ we denote by $\mathscr{G}^s(M)$ and $\mathscr{G}^{(s)}(M)$ the Gevrey classes of order $s$ on $M$ of Roumieu and Beurling type, respectively, and by $\mathsf{\Lambda}^1\mathscr{G}^s(M)$ and $\mathsf{\Lambda}^1\mathscr{G}^{(s)}(M)$ the corresponding spaces of Gevrey $1$-forms.
	
	Let $\omega_1,\dots,\omega_m$ be real-valued closed $1$-forms on $M$, belonging either to $\mathsf{\Lambda}^1\mathscr{G}^s(M)$ or to $\mathsf{\Lambda}^1\mathscr{G}^{(s)}(M)$. We study the operator
	\begin{equation}\label{L}
		\mathbb{L}:\mathscr{G}^{[s]}(M\times\mathbb{T}^m)\longrightarrow \mathsf{\Lambda}^1 \mathscr{G}^{[s]}(M\times\mathbb{T}^m),
		\qquad 
		\mathbb{L}u = \mathrm{d}_t u + \sum_{k=1}^{m} \omega_k \wedge \partial_{x_k} u,
	\end{equation}
	where $\mathbb{T}^m$ is the $m$-dimensional torus, $\mathrm{d}_t$ is the exterior derivative on $M$, and $[s]$ stands for $s$ or $(s)$. 
	
	Operators of the form \eqref{L} arise naturally in the study of tube-type involutive structures and may be viewed locally as systems of first-order linear partial differential equations. In the compact case, global properties of \eqref{L} have been extensively investigated, see, for instance, \cite{BCM1993,BCP1996,BP1999_jmaa,ADL2023gh,ADL2023gs,HZ2017,HZ2019,LM2024}. For the general theory of involutive systems we refer to \cite{BCH_book,Treves}.
	
	We recall the notions of global Gevrey hypoellipticity.
	
	\begin{definition}
		The operator $\mathbb{L}$ in \eqref{L} is called $[s]$-globally hypoelliptic if
			\[
			u \in \mathscr{D}'(M\times\mathbb{T}^m),\ \mathbb{L}u \in \mathsf{\Lambda}^1 \mathscr{G}^{[s]}(M\times\mathbb{T}^m)
			\ \Rightarrow\ 
			u \in \mathscr{G}^{[s]}(M\times\mathbb{T}^m).
			\]
	\end{definition}
	Here $\mathscr{D}'(M\times\mathbb{T}^m)$ denotes the space of distributions on $M\times\mathbb{T}^m$.
	
	Throughout the paper we assume $s>1$ and that $M$ is diffeomorphic to the interior of a compact manifold with boundary $\overline{M}$. 
	Let $\varrho:\overline{M}\to[0,+\infty)$ be a boundary defining function on $\overline{M}$. 
	A scattering metric on $M$ is a Riemannian metric $g$ which, in a collar neighbourhood of $\partial M$, has the form
	\begin{equation*}
		g = \frac{\mathrm{d}\varrho^2}{\varrho^4}+\frac{g'}{\varrho^2},
	\end{equation*}
	where $g'$ is a smooth symmetric $2$-cotensor restricting to a Riemannian metric on $\partial M$. In this case, $(M,g)$ is said to be a scattering manifold.
	
	The existence of a scattering metric is important in a key step: it provides a suitable version of the Hodge theorem (see \cite[Theorem~6.2]{Melrose_GST}), identifying the space of square-integrable harmonic $1$-forms with the image of the compactly supported de Rham cohomology inside $H^1_{\mathrm{dR}}(M)$. With this, in \cite{CKMT} the authors were able to extend the classic results for compact manifolds to this class of non-compact manifolds. By combining an analytic atlas with Gevrey partitions of unity (hence the restriction $s>1$), we construct a boundary defining function and a scattering metric whose coefficients are Gevrey of order $s$. The associated Laplace--Beltrami operator is then elliptic with Gevrey coefficients, which yields Gevrey regularity of harmonic forms and allows us to represent certain cohomology classes by Gevrey harmonic forms.
	
	In the compact case treated in \cite{ADL2023gh}, the authors rely on a theorem by Grauert \cite{Grauert} to endow $M$ with an analytic metric, which also permits the treatment of the case $s=1$. In the present non-compact setting, however, one needs a metric that simultaneously has Gevrey coefficients and is compatible with the scattering geometry near the boundary in order to apply the Hodge theorem, and
Grauert's theorem does not provide this additional structure. However, assuming the existence of an analytic scattering metric, our results naturally extend to the case $s=1$.
	
	Our analysis follows the strategy introduced in \cite{ADL2023gh}. Expanding $u$ in partial Fourier series on $\mathbb{T}^m$ reduces \eqref{L} to a family of twisted operators on $M$, and the obstruction to global Gevrey regularity is governed by small-denominator phenomena. These are encoded by the notions of rationality and of exponential Liouville behavior for the family $\boldsymbol{\omega}=(\omega_1,\dots,\omega_m)$, expressed in terms of periods along a basis of cycles associated with the image of compactly supported cohomology. We denote by $\mathsf{\Lambda}^1\mathscr{G}^{[s]}_{\partial M}(M)$ the space of real-valued closed Gevrey $1$-forms whose cohomology class lies in this distinguished subspace of $H^1_{\mathrm{dR}}(M)$.
	
	Our main result is the following characterization.
	
	\begin{theorem}
		Fix $s>1$ and let $\boldsymbol{\omega}=(\omega_1,\dots,\omega_m)$ be a family of real-valued closed $1$-forms in $\mathsf{\Lambda}^1\mathscr{G}^{[s]}_{\partial M}(M)$. 
		Then the operator $\mathbb{L}$ defined in \eqref{L} is $[s]$-globally hypoelliptic if and only if $\boldsymbol{\omega}$ is neither rational nor $[s]$-exponential Liouville.
	\end{theorem}
	
	The paper is organized as follows. In Section~\ref{sec-GSmanif} we recall the basic properties of Gevrey classes on open subsets of $\mathbb{R}^n$ and introduce the corresponding Roumieu and Beurling spaces on analytic manifolds, including Gevrey $1$-forms and the relevant boundedness criteria. In Section~\ref{gev_sc_met} we construct a Gevrey scattering metric on the interior of a compact manifold with boundary and derive Gevrey regularity for harmonic forms via elliptic regularity for the associated Laplace--Beltrami operator. Section~\ref{sec-Matrix-cycles} develops the cohomological framework needed in the non-compact setting: we define the space $\mathsf{\Lambda}^1\mathscr{G}^{[s]}_{\partial M}(M)$, associate to a family of closed forms its matrix of cycles, and relate rationality and exponential Liouville conditions to suitable Diophantine estimates. In Section~\ref{sec-sHypoel} we study the global Gevrey hypoellipticity of $\mathbb{L}$ using partial Fourier series on $\mathbb{T}^m$, reducing the problem to a family of twisted equations on $M$ and proving the characterization stated in the main theorem. Finally, for the reader's convenience, Appendix~\ref{appendix} collects the definitions and main results on partial Fourier series in Gevrey spaces used throughout the paper.

	\section{Gevrey spaces on manifolds}\label{sec-GSmanif}
	
	In this section, we introduce the notation and recall basic properties of Gevrey spaces on manifolds. We begin with the classical definition of Gevrey spaces on open subsets of $\mathbb{R}^n$.
	
	Let $\Omega \subset \mathbb{R}^n$ be an open set and fix $s \geq 1$. For each compact subset $K\subset\Omega$ (Notation: $K \Subset \Omega$) and each $h > 0$, we define the Banach space $\mathscr{G}^{s,h}(K)$ as the space of all functions $f \in C^\infty(\Omega)$ such that
	\[
	\|f\|_{K,h} := \sup_{\alpha \in \mathbb{N}_0^n} h^{-|\alpha|} \alpha!^{-s} 
	\sup_{x \in K} |\partial^\alpha f(x)| < \infty.
	\]
	
	We then define
	\[
	\mathscr{G}^s(K) := \bigcup_{h > 0} \mathscr{G}^{s,h}(K), 
	\qquad 
	\mathscr{G}^{(s)}(K) := \bigcap_{h > 0} \mathscr{G}^{s,h}(K),
	\]
	endowed with the inductive and projective limit topologies, respectively. These local constructions give rise to the global Roumieu and Beurling-type Gevrey spaces
	\[
	\mathscr{G}^s(\Omega) := \bigcap_{K \Subset \Omega} \mathscr{G}^s(K), 
	\qquad 
	\mathscr{G}^{(s)}(\Omega) := \bigcap_{K \Subset \Omega} \mathscr{G}^{(s)}(K),
	\]
	both endowed with the projective limit topology.
	
	\begin{proposition}
		Let $s \geq 1$. The spaces $\mathscr{G}^s(\Omega)$ and $\mathscr{G}^{(s)}(\Omega)$ are commutative algebras under pointwise multiplication and are stable under differentiation.
	\end{proposition}
	
	\begin{proof}
		See \cite[Proposition~1.4.5]{Rod_Gevrey} for the Roumieu case. The Beurling case follows by analogous arguments.
	\end{proof}
	
	\begin{proposition}
		If $1 \leq s_0 < s$, then $
		\mathscr{G}^{s_0}(\Omega) \subset \mathscr{G}^{(s)}(\Omega)$.
	\end{proposition}
	
	\begin{proof}
		Let $f \in \mathscr{G}^{s_0}(\Omega)$ and let $K \Subset \Omega$. Then there exist constants $C,H>0$ such that
		\[
		\sup_{t \in K} |\partial_t^\alpha f(t)| \leq C H^{|\alpha|}\alpha!^{s_0},
		\qquad \forall\, \alpha \in \mathbb{N}_0^n.
		\]
		
		Writing $\alpha!^{s_0} = \alpha!^s \alpha!^{-(s-s_0)}$ and using the inequalities
		\[
		\alpha! \geq n^{-|\alpha|}|\alpha|!, 
		\quad 
		|\alpha|! \geq \zeta^{|\alpha|} e^{-\zeta}, \qquad \forall\,\alpha\in\mathbb{N}_0^n,\ \forall,\zeta>0,
		\]
		we obtain
		\[
		\alpha!^{-(s-s_0)} \leq \left(\frac{n^{s-s_0}}{\zeta^{s-s_0}}\right)^{|\alpha|} e^{(s-s_0)\zeta},
		\qquad \forall\,\alpha\in\mathbb{N}_0^n,\ \forall\, \zeta>0.
		\]
		
		Hence,
		\[
		\sup_{t \in K} |\partial_t^\alpha f(t)|
		\leq C e^{(s-s_0)\zeta}
		\left(\frac{n^{s-s_0}H}{\zeta^{s-s_0}}\right)^{|\alpha|} \alpha!^s,\qquad\forall\,\alpha\in\mathbb{N}_0^n,\ \forall\,\zeta>0.
		\]
		
		Given $h>0$, choosing $
		\zeta = n(Hh^{-1})^{\frac{1}{s-s_0}}$, we obtain
		\[
		\frac{n^{s-s_0}H}{\zeta^{s-s_0}} = h.
		\]
		
		It follows that
		\[
		\sup_{t \in K} |\partial_t^\alpha f(t)| \leq C_0 h^{|\alpha|}\alpha!^s,\qquad\forall\,\alpha\in\mathbb{N}_0^n
		\]
		where $C_0 = C e^{(s-s_0)\zeta}$. Since $h>0$ is arbitrary, this shows that
		$f \in \mathscr{G}^{(s)}(\Omega)$.
	\end{proof}
	
	\begin{proposition}
		Let $\chi:\Omega \to \Omega'$ be a real-analytic mapping between open sets
		$\Omega \subset \mathbb{R}^n$ and $\Omega' \subset \mathbb{R}^m$.
		If $f \in \mathscr{G}^s(\Omega')$, then $f \circ \chi \in \mathscr{G}^s(\Omega)$.
		Moreover, if $f \in \mathscr{G}^{(s)}(\Omega')$, then
		$f \circ \chi \in \mathscr{G}^{(s)}(\Omega)$.
	\end{proposition}

	\begin{proof}
		The proof for the
		Roumieu case is classical and can be found in \cite[Proposition~1.4.6]{Rod_Gevrey}. For the sake of completeness and future references,
		we here provide a detailed proof for the Beurling case.
		
		Let $g \in \mathscr{G}^{(s)}(\Omega')$ and let $K \Subset \Omega$. Since $\chi$ is real-analytic, there exist constants $A_0,A_1>0$ such that
		\[
		\sup_{t\in K} |\partial_t^\gamma \chi(t)| \leq A_0 A_1^{|\gamma|} \gamma!,
		\qquad \forall\, \gamma \in \mathbb{N}_0^n.
		\]
		
		Moreover, since $\chi(K)$ is compact in $\Omega'$ and $g \in \mathscr{G}^{(s)}(\Omega')$, for every $h>0$ there exists $C_h>0$ such that
		\[
		\sup_{y\in \chi(K)} |\partial_y^\beta g(y)| \leq C_h h^{|\beta|} \beta!^s,
		\qquad \forall\, \beta \in \mathbb{N}_0^m.
		\]
		
		Fix $\alpha \in \mathbb{N}_0^n$ with $|\alpha|=k$. By the Faà di Bruno formula, there exists a constant $C>0$, depending only on $n$ and $k$, such that
		\[
		|\partial_t^\alpha (g\circ\chi)(t)|
		\leq C
		\sum_{1 \leq |\beta| \leq k}
		|\partial_t^\beta g(\chi(t))|
		\sum_{\substack{\gamma_1+\cdots+\gamma_{|\beta|}=\alpha\\ |\gamma_j|\geq 1}}
		\prod_{j=1}^{|\beta|}
		|\partial_t^{\gamma_j} \chi(t)|.
		\]
		
		Using the above estimates, we obtain
		\[
		|\partial_t^\alpha (g\circ\chi)(t)|
		\leq C C_h
		\sum_{1 \leq |\beta| \leq k}
		h^{|\beta|} \beta!^s
		\sum_{\substack{\gamma_1+\cdots+\gamma_{|\beta|}=\alpha}}
		\prod_{j=1}^{|\beta|}
		A_0 A_1^{|\gamma_j|} \gamma_j!.
		\]
		
		Since
		\[
		\prod_{j=1}^{|\beta|} \gamma_j! \leq \alpha!,
		\qquad
		\sum_{j=1}^{|\beta|} |\gamma_j| = |\alpha|,
		\]
		we obtain
		\[
		|\partial_t^\alpha (g\circ\chi)(t)|
		\leq C C_h
		(A_0)^{|\beta|}
		A_1^{|\alpha|}
		\alpha!
		\sum_{1 \leq |\beta| \leq k}
		h^{|\beta|} \beta!^s.
		\]
		
		Since $s \geq 1$, we have $\beta!^s \leq k!^s = |\alpha|!^s$. Therefore,
		\[
		|\partial_t^\alpha (g\circ\chi)(t)|
		\leq C C_h
		A_1^{|\alpha|}
		\alpha!^s
		\sum_{j=1}^{k} (A_0 h)^j.
		\]
		
		Choosing $h>0$ such that $A_0 h < 1$, we conclude that
		\[
		\sum_{j=1}^{k} (A_0 h)^j \leq \frac{A_0 h}{1-A_0 h},
		\]
		and hence
		\[
		|\partial_t^\alpha (g\circ\chi)(t)|
		\leq C_h'
		(A_1)^{|\alpha|}
		\alpha!^s.
		\]
		
		Since $h>0$ was arbitrary, this shows that $g\circ\chi \in \mathscr{G}^{(s)}(\Omega)$.
	\end{proof}
	
	For $s>1$, the spaces $\mathscr{G}_c^{s}(\Omega)$ and $\mathscr{G}_c^{(s)}(\Omega)$ of compactly supported Gevrey functions are nontrivial. The nontriviality of $\mathscr{G}_c^{s}(\Omega)$ for $s>1$ follows from \cite[Example~1.4.9]{Rod_Gevrey}. In the Beurling case, it suffices to observe that
	\[
	\mathscr{G}_c^{s_0}(\Omega)\subset \mathscr{G}_c^{(s)}(\Omega),
	\qquad 1<s_0<s.
	\]
	
	Now let $M$ be a smooth paracompact manifold. By a classical result of Whitney \cite{Whitney}, $M$ admits a countable and locally finite analytic atlas $
	\mathscr{A} = \{(\Omega_i,\chi_i)\}_{i\in I}$ which is compatible with its smooth structure.
	
	\begin{definition}
		Let $s\geq 1$. The space $\mathscr{G}^s(M)$ of Roumieu-type Gevrey functions on $M$ is defined as the space of all smooth functions $f$ on $M$ such that, for every analytic chart $(\Omega_i,\chi_i)$, the function $f\circ\chi_i^{-1}$ belongs to $\mathscr{G}^s(\chi_i(\Omega_i))$. Similarly, the space $\mathscr{G}^{(s)}(M)$ of Beurling-type Gevrey functions on $M$ consists of all smooth functions $f$ on $M$ such that $f\circ\chi_i^{-1}\in \mathscr{G}^{(s)}(\chi_i(\Omega_i))$ for every analytic chart $(\Omega_i,\chi_i)$.
	\end{definition}
	
	We shall use the notation $\mathscr{G}^{[s]}(M)$ to refer to either of the above cases, where ${[s]}$ stands for $s$ or $(s)$. If $s>1$, the spaces $\mathscr{G}_c^{[s]}(M)$ of compactly supported Gevrey functions on $M$ are well defined and nontrivial. As in the local case, if $1<s_0<s$, the following inclusions hold:
	\begin{equation}\label{inc_gev}
		\mathscr{G}^{(1)}(M)\subset
		\mathscr{G}^1(M)=C^\omega(M)\subset
		\mathscr{G}^{(s_0)}(M)\subset
		\mathscr{G}^{s_0}(M)\subset
		\mathscr{G}^{(s)}(M)\subset
		\mathscr{G}^s(M).
	\end{equation}
	
	Given a compact set $K \Subset M$, let $I_K \subset I$ be a finite subset of indices such that the coordinate domains $\{\Omega_i\}_{i\in I_K}$ cover $K$. Then one can choose compact sets $K_i \Subset \Omega_i$, $i\in I_K$, such that the family $\{K_i\}_{i\in I_K}$ still covers $K$.
	
	For each $h>0$, we define
	\begin{equation*}
		\|f\|_{K,h,M} := \sum_{i\in I_K} \|f\circ\chi_i^{-1}\|_{K_i,h}.
	\end{equation*}
	
	For each $h>0$, we set
	\[
	\mathscr{G}^{s,h}(K) := \{ f\in C^\infty(M) : \|f\|_{K,h,M}<\infty \}.
	\]
	
	Then,
	\[
	\mathscr{G}^s(K)
	= \bigcup_{h>0}\mathscr{G}^{s,h}(K)
	= \{ f\in C^\infty(M) : \|f\|_{K,h,M}<\infty \text{ for some } h>0 \},
	\]
	and
	\[
	\mathscr{G}^{(s)}(K)
	= \bigcap_{h>0}\mathscr{G}^{s,h}(K)
	= \{ f\in C^\infty(M) : \|f\|_{K,h,M}<\infty \text{ for all } h>0 \},
	\]
	endowed with the inductive and projective limit topologies, respectively.
	
	One can show that, for $h<h_+$, the inclusion mappings
	\[
	\mathscr{G}^{s,h}(K)\hookrightarrow \mathscr{G}^{s,h_+}(K)
	\]
	are compact. As a consequence, $\mathscr{G}^s(K)$ is a DFS space and
	$\mathscr{G}^{(s)}(K)$ is an FS space.
	
	The spaces $\mathscr{G}^s(M)$ and $\mathscr{G}^{(s)}(M)$ are then defined as the projective limits of $\mathscr{G}^s(K)$ and $\mathscr{G}^{(s)}(K)$, respectively, as $K$ ranges over all compact subsets of $M$. In particular, $\mathscr{G}^{(s)}(M)$ is an FS space, since it is the projective limit of FS spaces; see \cite{Komatsu1967}.
	
	A sequence $\{f_k\}_{k\in\mathbb{N}}$ in $\mathscr{G}^s(M)$ converges to $f\in\mathscr{G}^s(M)$ if and only if, for every $K\Subset M$, there exists $h>0$ such that
	\[
	\|f_k-f\|_{K,h,M} \to 0,\qquad k\to\infty.
	\]
	
	Similarly, a sequence $\{f_k\}_{k\in\mathbb{N}}$ in $\mathscr{G}^{(s)}(M)$ converges to $f\in\mathscr{G}^{(s)}(M)$ if and only if, for every $K\Subset M$ and every $h>0$,
	\[
	\|f_k-f\|_{K,h,M} \to 0,\qquad k\to\infty.
	\]
	
	As a consequence of the definition of the projective limit topology, a subset
	$B\subset \mathscr{G}^s(M)$ is bounded if and only if, for every $K\Subset M$, there exists $h>0$ such that $B$ is bounded in $\mathscr{G}^{s,h}(K)$. Similarly, a subset $B\subset \mathscr{G}^{(s)}(M)$ is bounded if and only if, for every $K\Subset M$ and every $h>0$, the set $B$ is bounded in $\mathscr{G}^{s,h}(K)$.
	
	We now introduce the spaces $\mathsf{\Lambda}^1\mathscr{G}^{[s]}(M)$ of Gevrey $1$-forms on $M$. In each local chart $(\Omega_i,\chi_i)$, a smooth $1$-form $f$ can be written as
	\[
	(f\circ\chi_i)(t) = \sum_{j=1}^n f_{ij}(t)\,\mathrm{d}t_{ij},
	\]
	where $(t_{i1},\dots,t_{in})$ are the local coordinates associated with $\chi_i$, and the coefficients $f_{ij}$ are smooth functions on $\chi_i(\Omega_i)$. If $(\Omega_i,\chi_i)$ is an analytic chart, then $f$ belongs to $\mathsf{\Lambda}^1\mathscr{G}^{[s]}(M)$ if and only if
	\[
	f_{ij}\in\mathscr{G}^{[s]}(\chi_i(\Omega_i))
	\qquad \forall\, i\in I,\ j=1,\dots,n.
	\]
	
	In each analytic chart $(\Omega_i,\chi_i)$, a Gevrey $1$-form can thus be identified with an element of $\big(\mathscr{G}^{[s]}(\chi_i(\Omega_i))\big)^n$. As before, given a compact set $K\Subset M$ and $h>0$, and choosing compact sets $K_i\Subset\Omega_i$ for $i\in I_K$, we define seminorms on $\mathsf{\Lambda}^1\mathscr{G}^{[s]}(M)$ by
	\[
	\|f\|_{K,h,\mathsf{\Lambda}^1}
	:= \sum_{i\in I_K}\sum_{j=1}^n \|f_{ij}\|_{K_i,h},
	\qquad h>0.
	\]
	
	More generally, the same construction applies to define Gevrey spaces of sections of analytic vector bundles over analytic manifolds. As a consequence of the definition of the above topologies, bounded sets in these spaces can be characterized as follows.
	
	\begin{proposition}
		A subset $B\subset \mathsf{\Lambda}^1\mathscr{G}^s(M)$ is bounded if and only if, for every $K\Subset M$, there exist constants $h>0$ and $C>0$ such that
		\[
		\|f\|_{K,h,\mathsf{\Lambda}^1} \leq C,
		\qquad \forall\, f\in B.
		\]
		Similarly, a subset $B\subset \mathsf{\Lambda}^1\mathscr{G}^{(s)}(M)$ is bounded if and only if, for every $K\Subset M$ and every $h>0$, there exists a constant $C>0$ such that
		\[
		\|f\|_{K,h,\mathsf{\Lambda}^1} \leq C,
		\qquad \forall\, f\in B.
		\]
	\end{proposition}

	\section{A Gevrey scattering metric}\label{gev_sc_met}
	
	Let $\overline{M}$ be a compact smooth manifold of dimension $n$ with boundary $\partial M$, and let $M$ denote its interior. In this section, for a fixed $s>1$, we construct a scattering metric $g$ on $M$ with Gevrey coefficients of order $s$ in every chart of a suitable analytic atlas. This construction relies on the existence of Gevrey partitions of unity.
	
	The existence of Gevrey partitions of unity of Roumieu type follows from the construction of compactly supported Gevrey functions on open subsets of $\mathbb{R}^n$ (see \cite[Example~1.4.9]{Rod_Gevrey}). The inclusions \eqref{inc_gev} imply the existence of Gevrey partitions of unity of Beurling type as well.
	
	\begin{definition}
		A \emph{boundary defining function} is a smooth function $\varrho:\overline{M}\to[0,+\infty)$ such that
		$\varrho^{-1}(0)=\partial M$ and $\mathrm{d}\varrho|_y \neq 0$ for all $y\in\partial M$.
	\end{definition}
	
	Near the boundary, $\overline{M}$ admits a collar neighbourhood diffeomorphic to $[0,1)_{\varrho}\times\partial M$. A Riemannian metric $g$ on $M$ is called a \emph{scattering metric} if, in such a collar neighbourhood, it takes the form
	\[
	g = \frac{\mathrm{d}\varrho^2}{\varrho^4} + \frac{g'}{\varrho^2},
	\]
	where $g'$ is a smooth symmetric $2$-cotensor that restricts to a Riemannian metric on $\partial M$. In this case, $(M,g)$ is called a scattering manifold. For further details on scattering geometry and analysis, we refer the reader to \cite{Melrose_APS,Melrose_SST,Melrose_GST}.

	We assume that $\overline{M}$ is endowed with a finite analytic atlas. This assumption is justified by the fact that the double $2M$ (obtained by gluing two copies of $\overline{M}$ along $\partial M$) is a closed manifold which we can endow with an analytic atlas due to Whitney \cite{Whitney}, and its analytic structure induces a compatible analytic atlas on $\overline{M}$.
	
	\begin{proposition}\label{bdf_gev}
		Let $\overline{M}$ be a compact manifold with boundary endowed with a finite analytic atlas $\mathscr{A}=\{(U_i,\phi_i)\}_{i=1}^N$, and let $s>1$. Then, there exists a boundary defining function $\varrho:\overline{M}\to[0,+\infty)$ which is Gevrey of order $s$ with respect to $\mathscr{A}$.
	\end{proposition}

	\begin{proof}
		In an interior chart $(U_i,\phi_i)$, we simply set $\varrho_i \equiv 1$. In a boundary chart $(U_i,\phi_i)$, with $\phi_i(U_i) \subset \mathbb{R}^{n-1} \times [0,+\infty)$, we define $\varrho_i(x_1,\dots,x_n) = x_n$. Then $\varrho_i = 0$ on $\partial M$, is positive away from $\partial M$,
		and $\displaystyle d\varrho_i|_{x_n=0}\not=0$. 
		Let $\{\psi_i\}_{i=1}^N$ be a Gevrey partition of unity of order $s$, of Roumieu or Beurling type, subordinated to the cover $\{U_i\}_{i=1}^N$. Since each $\varrho_i$ is analytic on its domain,
		\[
		\varrho(t) = \sum_{i=1}^N \psi_i(t)(\varrho_i \circ \phi_i)(t)
		\]
		is Gevrey of order $s$. Moreover, $\mathrm{d}\varrho$ does not vanish at any point $y \in \partial M$, as shown in the proof of \cite[Proposition 5.41]{Lee_smooth}.
	\end{proof}

	\begin{theorem}
		Let $\overline{M}$ be an $n$-dimensional smooth compact manifold with boundary, and let $s>1$. Then the interior $M$ admits a scattering metric whose coefficients are Gevrey of order $s$. We refer to such a metric as an $s$-scattering metric (in the Roumieu case) or an $(s)$-scattering metric (in the Beurling case).
	\end{theorem}
	
	\begin{proof}
		Let $\{(U_i,\phi_i)\}_{i=1}^N$ be a finite analytic atlas on $M$, and let $\varrho$ be the boundary defining function constructed in Proposition \ref{bdf_gev}. On each boundary chart $(U_i,\phi_i)$, the boundary
		\[
		\partial U_i = \{(x_1,\dots,x_n) \in U_i \,:\, x_n = 0\}
		\]
		is an analytically embedded submanifold, considering the usual analytic structure of $\mathbb{R}^n$. Hence, we can endow $\partial U_i$ with an analytic Riemannian metric $g_i'$, which may be regarded simply as the restriction to $\partial U_i$ of an analytic metric on $U_i$, which can be chosen to be the pullback of the Euclidean metric by $\phi_i$.
		
		Now, for each boundary chart $(U_i,\phi_i)$ on $M$, let $(V_i,\varphi_i)$ be the corresponding analytic chart on $2M$ that restricts to $(U_i,\phi_i)$. With respect to this cover of $2M$, let $\{\psi_i\}_{i=1}^N$ be a Gevrey partition of unity. By restricting this partition to $\partial M$ and using it to glue the metrics $g_i'$, we obtain a Gevrey metric $g'$ of order $s$ on $\partial M$.
		
		Finally, using the Gevrey boundary defining function $\varrho$, consider a collar neighbourhood $C \simeq [0,1) \times \partial M$ of the boundary. On the interior of $C$, we define the metric
		\[
		g = \dfrac{\mathrm{d}\varrho^2}{\varrho^4} + \dfrac{g'}{\varrho^2},
		\]
		which clearly has Gevrey coefficients. We also endow each interior chart with the standard analytic metric of $\mathbb{R}^n$. Finally, by using once again a Gevrey partition of unity, we obtain a globally defined scattering metric $g$ whose coefficients are Gevrey in every analytic chart.
	\end{proof}

	\begin{proposition}\label{gh_lapl}
		Fix $s>1$ and let $g$ be a ${[s]}$-scattering metric on the interior $M$ of a compact manifold with boundary $\overline{M}$. Then, the Laplace--Beltrami operator $\Delta_g$ is ${[s]}$-globally hypoelliptic on $M$.
	\end{proposition}

	\begin{proof}
		Let $\Delta_g$ denote the Laplace--Beltrami operator associated with the metric $g$. Since $g$ is a Riemannian metric on the open manifold $M$, $\Delta_g$ is a second-order elliptic differential operator on $M$. In any analytic local chart, its coefficients are Gevrey functions of order $s$ (of Roumieu or Beurling type, depending on the metric).
		
		In the Roumieu case, it follows from \cite[Theorem~3.3.8]{Rod_Gevrey} that elliptic differential operators with $s$-Gevrey coefficients are locally $s$-hypoelliptic. Similarly, in the Beurling case, \cite[Proposition~2.12]{FGJ2005} establishes local $(s)$-hypoellipticity for elliptic operators with $(s)$-Gevrey coefficients. These results are obtained on $\mathbb{R}^n$ by constructing a parametrix. Using a Gevrey partition of unity, one can use the local results to obtain a (Gevrey) global parametrix in the algebra $\Psi(M)=\bigcup_{\mu\in\mathbb{R}}\Psi^\mu(M)$ of pseudo-differential operators on $M$, which gives us the ${[s]}$-global hypoellipticity of $\Delta_g$.
	\end{proof}

	\begin{proposition}\label{gev_gh}
		Fix $s>1$. Let $M$ be the interior of a compact manifold with boundary, endowed with a ${[s]}$-scattering metric $g$. Then the exterior derivative $\mathrm{d}:\mathscr{G}^{[s]}(M)\to\mathsf{\Lambda}^1\mathscr{G}^{[s]}(M)$ is ${[s]}$-globally hypoelliptic. Furthermore, for any $\xi \in \mathbb{Z}^m$, the operator $\mathbb{L}_\xi:\mathscr{G}^{[s]}(M)\to\mathsf{\Lambda}^1\mathscr{G}^{[s]}(M)$ defined by
		\[
		\mathbb{L}_\xi u = \mathrm{d}u + i(\xi\cdot\boldsymbol{\omega})\wedge u
		\]
		is also ${[s]}$-globally hypoelliptic.
	\end{proposition}
	\begin{proof}
		The proof follows the same ideas as in \cite[Corollary~2.4]{CKMT}, observing that the corresponding Laplacians are elliptic and have Gevrey coefficients in every analytic chart of $M$.
	\end{proof}

	\section{Matrices of cycles}\label{sec-Matrix-cycles}
	
	Let $M$ be a smooth paracompact manifold endowed with an analytic atlas, as discussed in the previous section.
	
	\begin{definition}
		A smooth real-valued $1$-form $\omega$ on $M$ is said to be \emph{integral} if
		\[
		\frac{1}{2\pi}\int_\sigma \omega \in \mathbb{Z}
		\]
		for every smooth $1$-cycle $\sigma$ on $M$.
	\end{definition}
	
	\begin{definition}
		Fix $s>1$ and let $\boldsymbol{\omega}=(\omega_1,\dots,\omega_m)$ be a family of real-valued closed $1$-forms in $\mathsf{\Lambda}^1\mathscr{G}^{[s]}(M)$. We say that $\boldsymbol{\omega}$ is:
		\begin{enumerate}
			\item \emph{rational} if there exists $\xi\in\mathbb{Z}^m\setminus\{0\}$ such that
			\[
			\xi\cdot\boldsymbol{\omega} = \sum_{k=1}^{m}\xi_k\omega_k
			\]
			is an integral $1$-form;
			
			\item \emph{$s$-exponential Liouville} if $\boldsymbol{\omega}$ is not rational and there exist $\varepsilon>0$, a sequence of integral $1$-forms $\{\theta_j\}_{j\in\mathbb{N}}$ in $\mathsf{\Lambda}^1\mathscr{G}^s(M)$, and a sequence $\{\xi^{(j)}\}_{j\in\mathbb{N}}$ in $\mathbb{Z}^m\setminus\{0\}$ with $|\xi^{(j)}|\to\infty$, such that
			\[
			\big\{e^{\varepsilon|\xi^{(j)}|^{1/s}}(\xi^{(j)}\cdot\boldsymbol{\omega}-\theta_j)\big\}_{j\in\mathbb{N}}
			\]
			is bounded in $\mathsf{\Lambda}^1\mathscr{G}^s(M)$;
			
			\item \emph{$(s)$-exponential Liouville} if $\boldsymbol{\omega}$ is not rational and there exist a sequence of integral $1$-forms $\{\theta_j\}_{j\in\mathbb{N}}$ in $\mathsf{\Lambda}^1\mathscr{G}^{(s)}(M)$, and a sequence $\{\xi^{(j)}\}_{j\in\mathbb{N}}$ in $\mathbb{Z}^m\setminus\{0\}$ with $|\xi^{(j)}|\to\infty$, such that, for every $\varepsilon>0$,
			\[
			\big\{e^{\varepsilon|\xi^{(j)}|^{1/s}}(\xi^{(j)}\cdot\boldsymbol{\omega}-\theta_j)\big\}_{j\in\mathbb{N}}
			\]
			is bounded in $\mathsf{\Lambda}^1\mathscr{G}^{(s)}(M)$.
		\end{enumerate}
	\end{definition}
	
	Now let $\overline{M}$ be a compact manifold with boundary endowed with a finite analytic atlas, and let $M$ be its interior equipped with a ${[s]}$-scattering metric $g$.
	
	Let $H^1_{\mathrm{dR}}(\overline{M},\partial M)$ denote the first relative de Rham cohomology space, which is naturally isomorphic to the compactly supported cohomology $H^1_c(M)$. Consider the natural homomorphism
	\[
	\iota^*: H^1_{\mathrm{dR}}(\overline{M},\partial M)\to H^1_{\mathrm{dR}}(\overline{M}),
	\]
	and denote its image by $H^1_{\partial M}(M)$.
	
	Consider the space of square-integrable harmonic $1$-forms on $M$,
	\[
	\mathcal{H}^1(M) = \left\{ \omega \in \mathsf{\Lambda}^1 L^2(M) : \Delta \omega = 0 \right\},
	\]
	where $\Delta$ denotes the Gevrey scattering Laplacian constructed in Section~\ref{gev_sc_met}. By Proposition~\ref{gh_lapl}, the operator $\Delta$ is ${[s]}$-globally hypoelliptic. Hence, every harmonic form in $\mathcal{H}^1(M)$ belongs in fact to $\mathsf{\Lambda}^1 \mathscr{G}^{[s]}(M)$.
	
	The Hodge theorem for scattering metrics (see \cite[Theorem~6.2]{Melrose_GST}) establishes a canonical isomorphism
	\[
	\mathcal{H}^1(M) 	\;\simeq\; 	H^1_{\partial M}(M).
	\]
	
	It is well known that the inclusion $M \hookrightarrow \overline{M}$ induces an isomorphism
	\[
	H^1_{\mathrm{dR}}(\overline{M}) \;\simeq\;  H^1_{\mathrm{dR}}(M).
	\]
	
	Consequently, we may regard $H^1_{\partial M}(M)$ as a subspace of $H^1_{\mathrm{dR}}(M)$.
	
	Since $M$ is the interior of a compact manifold with boundary, the de Rham cohomology space $H^1_{\mathrm{dR}}(M)$ is finite-dimensional (see \cite[Remark 2.8]{CKMT}). Let $d' = \dim H^1_{\mathrm{dR}}(M)$ and $d=\dim H^1_{\partial M}(M)$.
	
	Let $\{[\omega_1],\dots,[\omega_d]\}$ be a basis of the subspace $H^1_{\partial M}(M)$, and complete it to a basis
	\[
	\mathcal{B}' = \{[\omega_1],\dots,[\omega_d],[\omega_{d+1}],\dots,[\omega_{d'}]\}
	\]
	of the full space $H^1_{\mathrm{dR}}(M)$.
	
	By the de Rham Theorem, there exists a family of $1$-cycles $\{\sigma_1,\dots,\sigma_{d'}\}	\subset M$ whose homology classes form a basis of $H_1(M;\mathbb{R})$ dual to $\mathcal{B}'$ with respect to the natural pairing 
	\[
	H^1_{\mathrm{dR}}(M)\times H_1(M;\mathbb{R}) \longrightarrow \mathbb{R}, \qquad ([\omega],[\sigma]) \longmapsto \int_\sigma \omega.
	\]
	
	We denote by $\mathsf{\Lambda}^1\mathscr{G}^{[s]}_{\partial M}(M)$ the space of real-valued closed $1$-forms
	$\omega \in \mathsf{\Lambda}^1\mathscr{G}^{[s]}(M)$ whose cohomology class lies in
	$H^1_{\partial M}(M)$.
	
	By construction of the dual basis, if $\omega \in \mathsf{\Lambda}^1\mathscr{G}^{[s]}_{\partial M}(M)$, then
	\[
	\int_{\sigma_\ell} \omega = 0, \qquad d < \ell \le d'.
	\]
	
	Thus, the rationality of such a form depends only on its periods along 	$\sigma_1, \dots, \sigma_d$.
	
	Given a family $\boldsymbol{\omega}=(\omega_1,\dots,\omega_m)$ of $1$-forms in 
	$\mathsf{\Lambda}^1\mathscr{G}^{[s]}_{\partial M}(M)$, we associate to $\boldsymbol{\omega}$
	its matrix of cycles $A(\boldsymbol{\omega}) \in M_{d\times m}(\mathbb{R})$ defined by
	\[
	A(\boldsymbol{\omega})_{\ell k} = \frac{1}{2\pi}\int_{\sigma_\ell}\omega_k, \qquad 	\ell=1,\dots,d,\ k=1,\dots,m.
	\]
	Equivalently, for every $\xi\in\mathbb{Z}^m$,
	\[
	A(\boldsymbol{\omega})\xi = \frac{1}{2\pi} 
	\left( \int_{\sigma_1}\xi\cdot\boldsymbol{\omega}, \dots, 	\int_{\sigma_d} \xi \cdot \boldsymbol{\omega}  \right)^T.
	\]
	
	Since the integrals depend only on the cohomology classes $[\omega_k]$,
	the matrix $A(\boldsymbol{\omega})$ depends only on the $m$-tuple
	$([\omega_1],\dots,[\omega_m]) \in (H^1_{\partial M}(M))^m$.
	Thus, the construction defines a linear map
	\[
	A : (H^1_{\partial M}(M))^m \longrightarrow M_{d\times m}(\mathbb{R}).
	\]
	
	\begin{definition}
		Let $s>1$. A matrix $\boldsymbol{A}\in M_{d\times m}(\mathbb{R})$
		is said to satisfy:
		\begin{enumerate}
			\item the condition $(\mathrm{DC}_s)$ if, for every $\varepsilon>0$,
			there exists a constant $C_\varepsilon>0$ such that
			\begin{equation}\label{DCs}
				|\eta+\boldsymbol{A}\xi|
				\geq
				C_\varepsilon
				e^{-\varepsilon (|\eta|+|\xi|)^{1/s}},
				\tag{DC$_s$}
			\end{equation}
			for all $(\xi,\eta)\in
			\mathbb{Z}^m\times\mathbb{Z}^d\setminus\{(0,0)\}$;
			
			\item the condition $(\mathrm{DC}_{(s)})$ if there exist constants
			$\varepsilon>0$ and $C>0$ such that
			\begin{equation}\label{DCss}
				|\eta+\boldsymbol{A}\xi|
				\geq
				C
				e^{-\varepsilon (|\eta|+|\xi|)^{1/s}},
				\tag{DC$_{(s)}$}
			\end{equation}
			for all $(\xi,\eta)\in
			\mathbb{Z}^m\times\mathbb{Z}^d\setminus\{(0,0)\}$.
		\end{enumerate}
	\end{definition}

	The following lemma relates the Diophantine condition $(\mathrm{DC}_s)$ to lower bounds on exponential sums on the torus.
	
	\begin{lemma}\label{est_DCs}
		Suppose that $\boldsymbol{A}\in M_{d\times m}(\mathbb{R})$ satisfies condition \eqref{DCs}. Then, for every $\varepsilon>0$, there exists a constant $C_\varepsilon>0$ such that
		\[
		\max_{1\leq \ell \leq d}
		\big|e^{2\pi i a_\ell\cdot \xi}-1\big|
		\geq
		C_\varepsilon
		e^{-\varepsilon |\xi|^{1/s}},
		\]
		for all $\xi\in\mathbb{Z}^m\setminus\{0\}$,
		where $a_\ell$ denotes the $\ell$-th row of $\boldsymbol{A}$.
	\end{lemma}
	
	The proof follows from the argument in \cite[Lemma~8.1]{DM2020}. 
	Using a similar approach, we establish the Beurling version of this result.
		
	\begin{lemma}\label{est_DCss}
		Suppose that $\boldsymbol{A}\in M_{d\times m}(\mathbb{R})$ satisfies condition \eqref{DCss}. Then there exist constants $\varepsilon,C>0$ such that
		\begin{equation}\label{equiv_DCss}
			\max_{1\leq\ell\leq d}\left|e^{2\pi i a_\ell\cdot\xi}-1\right|\geq C e^{-\varepsilon|\xi|^{\frac{1}{s}}},
		\end{equation}
		for all $\xi\in\mathbb{Z}^m\setminus\{0\}$, where $a_\ell$ denotes the $\ell$-th row of $\boldsymbol{A}$.
	\end{lemma}
	
	\begin{proof}
		Suppose that \eqref{equiv_DCss} does not hold. Then there exists a sequence $\{\xi^{(j)}\}_{j\in\mathbb{N}}$ in $\mathbb{Z}^m$ such that $|\xi^{(j+1)}|>|\xi^{(j)}|$ and
		\begin{equation}\label{est_max_b} 
			\max_{1\leq\ell\leq d}\left|e^{2\pi i a_\ell\cdot\xi^{(j)}}-1\right|<e^{-\frac{1}{j}|\xi^{(j)}|^{\frac{1}{s}}},
		\end{equation}
		for every $j\in\mathbb{N}$.
		
		For each $j\in\mathbb{N}$, consider $\eta^{(j)} = (\eta^{(j)}_1,\dots,\eta^{(j)}_d)\in\mathbb{Z}^d$, where each $\eta^{(j)}_{\ell}=-\lfloor a_\ell\cdot\xi^{(j)} \rfloor$. Here, $\lfloor \cdot \rfloor$ denotes the integer part of the corresponding real number. By \eqref{est_max_b}, we have
		\[
		\left|e^{2\pi i (a_\ell\cdot \xi^{(j)} + \eta^{(j)}_{\ell})}-1\right|
		=
		\left|e^{2\pi i a_\ell\cdot \xi^{(j)}}-1\right|
		\to 0,
		\qquad j\to \infty.
		\]
		
		Then the fractional part of $a_\ell\cdot\xi^{(j)}$ converges to $0$ or to $1$. Hence, we obtain
		\[
		\pi\left|a_\ell\cdot\xi^{(j)}+\eta^{(j)}_{\ell}\right|
		\leq 
		\left|e^{2\pi i(a_\ell\cdot\xi^{(j)}+\eta^{(j)}_{\ell})}-1\right|
		<
		e^{-\frac{1}{j}|\xi^{(j)}|^{\frac{1}{s}}},
		\]
		for every $j\in\mathbb{N}$ and $\ell=1,\dots,d$. We conclude that \eqref{DCss} does not hold for any $\varepsilon>0$.
	\end{proof}

	\begin{proposition}
			Let $\boldsymbol{\omega}=(\omega_1,\dots,\omega_m)$ be a family of $1$-forms in 
			$\mathsf{\Lambda}^1\mathscr{G}^{[s]}_{\partial M}(M)$. Then, $\boldsymbol{\omega}$ is rational if and only if 
			$A(\boldsymbol{\omega})(\mathbb{Z}^m\setminus\{0\})\cap\mathbb{Z}^d\neq\varnothing$.
	\end{proposition}
	\begin{proof}
		See \cite[Proposition 3.8 (i)]{CKMT}.
	\end{proof}
	
	\begin{proposition}\label{liou_DCs}
		Let $\boldsymbol{\omega}=(\omega_1,\dots,\omega_m)$ be a family of $1$-forms in 
		$\mathsf{\Lambda}^1\mathscr{G}^{[s]}_{\partial M}(M)$. Then, the system $\boldsymbol{\omega}$ is:
		\begin{enumerate}
			\item[(i)] $s$-exponential Liouville if and only if it is not rational and 
			$A(\boldsymbol{\omega})$ does not satisfy \eqref{DCs};
			
			\item[(ii)] $(s)$-exponential Liouville if and only if it is not rational and 
			$A(\boldsymbol{\omega})$ does not satisfy \eqref{DCss}.
		\end{enumerate}
	\end{proposition}

	\begin{proof}
	By the Hodge theorem for scattering manifolds, there exist 
	$\vartheta_1,\dots,\vartheta_d \in \mathcal{H}^1(M)$ whose cohomology classes form a basis of 
	$H^1_{\partial M}(M)$ dual to the cycles $\sigma_1,\dots,\sigma_d$. 
	Accordingly, for each $k=1,\dots,m$, we can write
	\begin{equation}\label{decomp_omega_k}
		\omega_k = \sum_{\ell=1}^d \lambda_{\ell k}\,\vartheta_\ell
		+ \mathrm{d}v_k,
	\end{equation}
	for suitable functions $v_k$.  By Propositions~\ref{gev_gh} and~\ref{gh_lapl}, we have $v_k\in\mathscr{G}^{[s]}(M)$ and $\vartheta_\ell \in\mathsf{\Lambda}^1 \mathscr{G}^{[s]}(M)$.
	
	Integrating \eqref{decomp_omega_k} over $\sigma_\ell$ yields
	\[
	\lambda_{\ell k} = A(\boldsymbol{\omega})_{\ell k}, \qquad \ell=1,\dots,d,\ k=1,\dots,m.
	\]
	
	Consequently, for every $\xi\in\mathbb{Z}^m$,
	\begin{equation}\label{omega}
		\xi\cdot\boldsymbol{\omega} = \sum_{\ell=1}^d \left( \sum_{k=1}^m A(\boldsymbol{\omega})_{\ell k}\xi_k \right)\vartheta_\ell + \sum_{k=1}^m \xi_k\,\mathrm{d}v_k.
	\end{equation}
	
	\medskip
	
	\noindent
	$(i)$ Suppose that $A(\boldsymbol{\omega})$ does not satisfy \eqref{DCs}. 
	Then there exist $\varepsilon>0$ and a sequence
	\(\{(\eta^{(j)},\xi^{(j)})\}\) in \(\mathbb{Z}^d \times (\mathbb{Z}^m \setminus \{0\})\) such that
	\begin{equation}\label{not_DC_r}
		|\eta^{(j)} + A(\boldsymbol{\omega})\xi^{(j)}| < j^{-1} e^{-\varepsilon |\xi^{(j)}|^{1/s}}, \qquad j\in\mathbb{N}.
	\end{equation}
	
	We claim that $|\xi^{(j)}|\to\infty$. 
	Indeed, if $\{\xi^{(j)}\}_{j\in\mathbb{N}}$ is bounded, then $A(\boldsymbol{\omega})\xi^{(j)}$ would take only finitely many values. 
	Since
	\[
	|\eta^{(j)}| \le |\eta^{(j)} + A(\boldsymbol{\omega})\xi^{(j)}| + |A(\boldsymbol{\omega})\xi^{(j)}|,
	\]
	the sequence $\{\eta^{(j)}\}_{j\in\mathbb{N}}$ would also be bounded, and hence $(\eta^{(j)},\xi^{(j)})$ would assume only finitely many values, contradicting \eqref{not_DC_r}. 
	Thus $|\xi^{(j)}|\to\infty$.
	
	For each $j\in\mathbb{N}$, define
	\begin{equation}\label{vartheta_r}
		\theta_j = -\sum_{\ell=1}^d \eta^{(j)}_\ell\,\vartheta_\ell + \sum_{k=1}^m \xi^{(j)}_k\,\mathrm{d}v_k \in		\mathsf{\Lambda}^1 \mathscr{G}^s(M),
	\end{equation}
	where
	\[\xi^{(j)}=(\xi^{(j)}_1,\dots,\xi^{(j)}_m)\quad\text{and}\quad \eta^{(j)}=(\xi^{(j)}_1,\dots,\eta^{(j)}_d).\]
	
	Each $\theta_j$ is integral, since $(\eta^{(j)})_\ell\in\mathbb{Z}$, 
	the forms $\vartheta_\ell$ are integral, and exact forms are integral.
	
	Using \eqref{omega} and \eqref{vartheta_r}, we obtain
	\begin{align*}
		\rho_j & := e^{\varepsilon|\xi^{(j)}|^{\frac{1}{s}}}(\xi^{(j)} \cdot \boldsymbol{\omega} - \theta_j)\\
		& = e^{\varepsilon|\xi^{(j)}|^{\frac{1}{s}}}\left[\sum_{\ell=1}^{d}\left(\sum_{k=1}^{m}A(\boldsymbol{\omega})_{\ell k}\xi^{(j)}_k\right)\vartheta_\ell + \sum_{k=1}^{m}\xi^{(j)}_k\,\mathrm{d}v_k + \sum_{\ell=1}^{d}\eta^{(j)}_\ell\vartheta_\ell - \sum_{k=1}^{m}\xi^{(j)}_k\,\mathrm{d}v_k\right]\\
		& = e^{\varepsilon|\xi^{(j)}|^{\frac{1}{s}}} \sum_{\ell=1}^{d} \left(\sum_{k=1}^{m} A(\boldsymbol{\omega})_{\ell k}\xi^{(j)}_k + \eta^{(j)}_\ell\right)\vartheta_\ell.
	\end{align*}

	It remains to prove that the sequence $\{\rho_j\}_{j\in\mathbb{N}}$ is bounded in 	$\mathsf{\Lambda}^1 \mathscr{G}^s(M)$. Let $(V,t)$ be a local chart and let $K\Subset V$.  For each multi-index $\alpha\in\mathbb{N}_0^n$, we have 
	\begin{equation}\label{sup_dtrho_r}
		\sup_{t\in K} |\partial_t^\alpha \rho_j(t)| 
		= e^{\varepsilon|\xi^{(j)}|^{1/s}} \sum_{\ell=1}^d 	\left| \sum_{k=1}^m
		A(\boldsymbol{\omega})_{\ell k}\xi^{(j)}_k + \eta^{(j)}_\ell \right|
		\, 	\sup_{t\in K} |\partial_t^\alpha \vartheta_\ell(t)|.
	\end{equation}
	
	Since each $\vartheta_\ell$ is Gevrey of order $s$, 
	there exist constants $C,h>0$, depending only on $K$, such that
	\begin{equation}\label{est_harm_r}
		\sup_{t\in K}	|\partial_t^\alpha \vartheta_\ell(t)|
		\le C h^{|\alpha|}\alpha!^s, \qquad \alpha\in\mathbb{N}_0^n,\ \ell=1,\dots,d.
	\end{equation}
	
	Substituting \eqref{est_harm_r} into \eqref{sup_dtrho_r}, we obtain
	\begin{align*}
		\sup_{t\in K} |\partial_t^\alpha \rho_j(t)|
		&\le C h^{|\alpha|}\alpha!^s e^{\varepsilon|\xi^{(j)}|^{1/s}} \sum_{\ell=1}^d
		\left| \sum_{k=1}^m A(\boldsymbol{\omega})_{\ell k}\xi^{(j)}_k + \eta^{(j)}_\ell 	\right|.
	\end{align*}
	
	Observe that
	\[
	\sum_{\ell=1}^d \left| \sum_{k=1}^m A(\boldsymbol{\omega})_{\ell k}\xi^{(j)}_k +
	\eta^{(j)}_\ell \right|
	\le C_1 |\eta^{(j)} + A(\boldsymbol{\omega})\xi^{(j)}|, 
	\]
	for some constant $C_1>0$ depending only on $d$.  Hence,
	\[
	\sup_{t\in K} |\partial_t^\alpha \rho_j(t)| \le C' h^{|\alpha|}\alpha!^s 
	e^{\varepsilon|\xi^{(j)}|^{1/s}} |\eta^{(j)} + A(\boldsymbol{\omega})\xi^{(j)}|.
	\]
	
	By \eqref{not_DC_r}, the sequence
	\[
	\big\{ e^{\varepsilon|\xi^{(j)}|^{1/s}} |\eta^{(j)} + A(\boldsymbol{\omega})\xi^{(j)}|	\big\}_{j\in\mathbb{N}}
	\]
	is bounded in $\mathbb{R}$. Therefore there exists $C''>0$, independent of $j\in\mathbb{N}$, such that
	\[
	\sup_{t\in K} |\partial_t^\alpha \rho_j(t)| 
	\le C'' h^{|\alpha|}\alpha!^s, \qquad \alpha\in\mathbb{N}_0^n.
	\]
	
	This shows that $\{\rho_j\}_{j\in\mathbb{N}}$ is bounded in $\mathsf{\Lambda}^1\mathscr{G}^s(M)$.

	\medskip
	
	Conversely, suppose that the family $\boldsymbol{\omega}$ is $s$-exponential Liouville.  Then there exist $\varepsilon>0$, a sequence $\{\xi^{(j)}\}_{j\in\mathbb{N}} $ in $\mathbb{Z}^m$ with $|\xi^{(j)}|\to\infty$, and a sequence of integral $1$-forms $\{\theta_j\}_{j\in\mathbb{N}}$ such that the sequence
	\[
	\{\rho_j\}_{j\in\mathbb{N}} := \big\{e^{\varepsilon|\xi^{(j)}|^{1/s}} (\xi^{(j)} \cdot \boldsymbol{\omega} - \theta_j)\big\}_{j\in\mathbb{N}}
	\]
	is bounded in $\mathsf{\Lambda}^1 \mathscr{G}^s(M)$.
	
	Since each $\theta_j$ is integral and belongs to $\mathsf{\Lambda}^1\mathscr{G}^s(M)$, we may write
	\begin{equation}\label{theta_r}
		\theta_j = \sum_{\ell=1}^{d} \mu_{j\ell}\,\vartheta_\ell + \mathrm{d}w_j,
	\end{equation}
	with $\mu_{j\ell} \in \mathbb{Z}$ and $w_j \in \mathscr{G}^s(M)$.
	
	Combining \eqref{omega} and \eqref{theta_r}, we obtain
	\begin{align}
		\rho_j &= e^{\varepsilon|\xi^{(j)}|^{1/s}} \Bigg[ \sum_{\ell=1}^{d}
		\left(\sum_{k=1}^{m} A(\boldsymbol{\omega})_{\ell k}\xi^{(j)}_k
		- \mu_{j\ell} \right) \vartheta_\ell + \sum_{k=1}^{m}\xi^{(j)}_k\,\mathrm{d}v_k
		- \mathrm{d}w_j \Bigg]. \label{rho_j_r}
	\end{align}
	
	Since $\int_{\sigma_\ell} \mathrm{d}v_k = 0$ and $\int_{\sigma_\ell} \mathrm{d}w_j = 0$, 
	integration of \eqref{rho_j_r} over $\sigma_\ell$ yields
	\[
	\frac{1}{2\pi} \int_{\sigma_\ell} \rho_j = e^{\varepsilon|\xi^{(j)}|^{1/s}} \left( \sum_{k=1}^{m}
	A(\boldsymbol{\omega})_{\ell k}\xi^{(j)}_k - \mu_{j\ell} \right), \qquad \ell=1,\dots,d.
	\]

	Since $\{\rho_j\}_{j\in\mathbb{N}}$ is bounded in $\mathsf{\Lambda}^1 \mathscr{G}^s(M)$, there exists a constant $C>0$ such that
	\[
	\left| \int_{\sigma_\ell} \rho_j \right| \le C, \qquad \forall\, j\in\mathbb{N},\ \ell=1,\dots,d.
	\]
	
	Therefore,
	\[
	e^{\varepsilon|\xi^{(j)}|^{1/s}} \left| \sum_{k=1}^{m} A(\boldsymbol{\omega})_{\ell k}\xi^{(j)}_k - \mu_{j\ell} \right| \le C', \qquad \ell=1,\dots,d,
	\]
	for some $C'>0$ independent of $j$.
	
	Defining $\eta^{(j)} := -(\mu_{j1}, \dots, \mu_{jd}) \in \mathbb{Z}^d$, we obtain
	\[
	|\eta^{(j)} + A(\boldsymbol{\omega})\xi^{(j)}| \le C'' e^{-\varepsilon|\xi^{(j)}|^{1/s}},
	\]
	for some $C''>0$ independent of $j$.
	
	Since $|\xi^{(j)}| \to \infty$, passing to a subsequence if necessary, we can assume that $e^{-\frac{\varepsilon}{2}|\xi^{(j)}|^{\frac{1}{s}}} \leq (C''j)^{-1}$. Hence,
	\[
	|\eta^{(j)} + A(\boldsymbol{\omega}) \xi^{(j)}| \leq j^{-1} e^{-\frac{\varepsilon}{2}|\xi^{(j)}|^{\frac{1}{s}}}, \qquad \forall\, j \in \mathbb{N},
	\]
	which implies that \eqref{DCs} does not hold.
	
	\medskip
	
	\noindent
	$(ii)$ Suppose first that $A(\boldsymbol{\omega})$ does not satisfy \eqref{DCss}. 
	Then, for every $j\in\mathbb{N}$, there exist  	$(\eta^{(j)},\xi^{(j)})$ in $\mathbb{Z}^d \times (\mathbb{Z}^m\setminus\{0\})$ such that
	\begin{equation}\label{not_DC_b}
		|\eta^{(j)} + A(\boldsymbol{\omega})\xi^{(j)}| 	< e^{-j|\xi^{(j)}|^{1/s}}.
	\end{equation}
	
	As in the proof of (i) above, one verifies that $\{|\xi^{(j)}|\}$ is unbounded. Hence, passing to a subsequence if necessary, we may assume that $|\xi^{(j)}|\to\infty$.
	
	For each $j$, define
	\begin{equation*}
		\theta_j = -\sum_{\ell=1}^d \eta^{(j)}_\ell\,\vartheta_\ell + \sum_{k=1}^m \xi^{(j)}_k\,\mathrm{d}v_k,
	\end{equation*}
	where
	\[\xi^{(j)}=(\xi^{(j)}_1,\dots,\xi^{(j)}_m)\quad\text{and}\quad \eta^{(j)}=(\xi^{(j)}_1,\dots,\eta^{(j)}_d)\]
	as before.	
	Each $\theta_j$ is integral, since $\eta^{(j)}_\ell\in\mathbb{Z}$, $\ell=1,\dots,d$, and exact forms are integral.

	Fix $\varepsilon>0$. We compute
	\begin{align*}
		\rho_j^{(\varepsilon)} 
		&:= e^{\varepsilon|\xi^{(j)}|^{\frac{1}{s}}} (\xi^{(j)} \cdot \boldsymbol{\omega} - \theta_j) \\
		&=
		e^{\varepsilon|\xi^{(j)}|^{\frac{1}{s}}} \left[ \sum_{\ell=1}^{d} \left(
		\sum_{k=1}^{m} A(\boldsymbol{\omega})_{\ell k}\xi^{(j)}_k \right)\vartheta_\ell
		+ \sum_{k=1}^{m}\xi^{(j)}_k\,\mathrm{d}v_k + 	\sum_{\ell=1}^{d} \eta^{(j)}_\ell\vartheta_\ell - \sum_{k=1}^{m}\xi^{(j)}_k\,\mathrm{d}v_k \right]\\
		& = e^{\varepsilon|\xi^{(j)}|^{\frac{1}{s}}} \sum_{\ell=1}^{d}
		\left( 	\sum_{k=1}^{m} A(\boldsymbol{\omega})_{\ell k}\xi^{(j)}_k + \eta^{(j)}_\ell
		\right)\vartheta_\ell. 
	\end{align*}

	We now prove that the sequence $\{\rho_j^{(\varepsilon)}\}_{j\in\mathbb{N}}$ is bounded in 
	$\mathsf{\Lambda}^1\mathscr{G}^{(s)}(M)$.
	
	Let $(V,t)$ be a local chart and let $K\Subset V$. 
	Since each $\vartheta_\ell$ is Gevrey of Beurling type of order $s$, 
	for every $h>0$ there exists $C_h>0$ such that
	\[
	\sup_{t\in K} |\partial_t^\alpha \vartheta_\ell(t)|
	\le C_h\, h^{|\alpha|}\alpha!^s, \qquad \alpha\in\mathbb{N}_0^n,
	\ \ell=1,\dots,d.
	\]
	
	Therefore, for every $\alpha\in\mathbb{N}_0^n$, we have
	\begin{align*}
		\sup_{t\in K} |\partial_t^\alpha \rho_j^{(\varepsilon)}(t)|
		&\le e^{\varepsilon|\xi^{(j)}|^{\frac{1}{s}}} \sum_{\ell=1}^{d} \left|
		\sum_{k=1}^{m} A(\boldsymbol{\omega})_{\ell k}\xi^{(j)}_k  + \eta^{(j)}_\ell \right| \sup_{t\in K} |\partial_t^\alpha \vartheta_\ell(t)| \\
		&\le C_h\, h^{|\alpha|}\alpha!^s e^{\varepsilon|\xi^{(j)}|^{\frac{1}{s}}} 		\sum_{\ell=1}^{d} \left| \sum_{k=1}^{m} A(\boldsymbol{\omega})_{\ell k}\xi^{(j)}_k
		+ \eta^{(j)}_\ell \right|.
	\end{align*}
	
	Since
	\[
	\sum_{\ell=1}^{d} \left| \sum_{k=1}^{m} A(\boldsymbol{\omega})_{\ell k}\xi^{(j)}_k
	+ \eta^{(j)}_\ell \right| \le C_1\, |\eta^{(j)} + A(\boldsymbol{\omega})\xi^{(j)}| 
	\]
	for some constant $C_1>0$ depending only on $d$, 
	we obtain
	\[
	\sup_{t\in K} |\partial_t^\alpha \rho_j^{(\varepsilon)}(t)|  \le C_h'\, h^{|\alpha|}\alpha!^s e^{\varepsilon|\xi^{(j)}|^{\frac{1}{s}}} |\eta^{(j)} + A(\boldsymbol{\omega})\xi^{(j)}|.
	\]
	
	Using \eqref{not_DC_b}, we deduce
	\[
	\sup_{t\in K} |\partial_t^\alpha \rho_j^{(\varepsilon)}(t)|
	\le 	C_h'\, h^{|\alpha|}\alpha!^s e^{(\varepsilon-j)|\xi^{(j)}|^{\frac{1}{s}}}.
	\]
	
	If $j>\varepsilon$, then $e^{(\varepsilon-j)|\xi^{(j)}|^{\frac{1}{s}}}\le 1$, and hence
	\[
	\sup_{t\in K} |\partial_t^\alpha \rho_j^{(\varepsilon)}(t)|
	\le C_h''\, h^{|\alpha|}\alpha!^s,
	\]
	uniformly in $j$, for $j$ sufficiently large.
	
	Since $\varepsilon>0$ was arbitrary and the above estimate holds for every $h>0$, 
	it follows that $\{\rho_j^{(\varepsilon)}\}_{j\in\mathbb{N}}$ is bounded in 
	$\mathsf{\Lambda}^1\mathscr{G}^{(s)}(M)$.	
	Therefore, $\boldsymbol{\omega}$ is $(s)$-exponential Liouville.
	
	Conversely, suppose that $\boldsymbol{\omega}$ is $(s)$-exponential Liouville. 
	Then there exists a sequence $\{\xi^{(j)}\}_{j\in\mathbb{N}}$ in $\mathbb{Z}^m$ with $|\xi^{(j)}|\to\infty$ 
	and integral $1$-forms $\theta_j$ such that, for every $\varepsilon>0$, 
	the sequence
	\[
	\{\rho_j^{(\varepsilon)} \}_{j\in\mathbb{N}} = \big\{e^{\varepsilon|\xi^{(j)}|^{1/s}} 	(\xi^{(j)}\cdot \boldsymbol{\omega} - \theta_j)\big\}_{j\in\mathbb{N}}
	\]
	is bounded in $\mathsf{\Lambda}^1\mathscr{G}^{(s)}(M)$.
	
	Writing
	\[
	\theta_j = \sum_{\ell=1}^d \mu_{j\ell}\,\vartheta_\ell + \mathrm{d}w_j,
	\qquad \mu_{j\ell}\in\mathbb{Z},
	\]
	and integrating over $\sigma_\ell$, we obtain
	\[
	\frac{1}{2\pi} \int_{\sigma_\ell} \rho_j^{(\varepsilon)} 
	= e^{\varepsilon|\xi^{(j)}|^{1/s}} \left( A(\boldsymbol{\omega})\xi^{(j)} - \mu_j \right)_\ell.
	\]
	
	Since $\{\rho_j^{(\varepsilon)}\}_{j\in\mathbb{N}}$ is bounded in $\mathsf{\Lambda}^1 \mathscr{G}^{(s)}(M)$, for every $\varepsilon>0$ there exists $C_\varepsilon>0$ such that
	\[
	|\eta^{(j)} + A(\boldsymbol{\omega})\xi^{(j)}| 	\le C_\varepsilon e^{-\varepsilon|\xi^{(j)}|^{1/s}},
	\qquad \eta^{(j)}:=-\mu_j\in\mathbb{Z}^d.
	\]
	
	As this estimate holds for every $\varepsilon>0$,  condition \eqref{DCss} cannot be satisfied. 
	
	Conversely, suppose that the family $\boldsymbol{\omega}$ is $(s)$-exponential Liouville. 	Then there exist a sequence $\{\xi^{(j)}\}_{j\in\mathbb{N}}$ in $\mathbb{Z}^m$, with $|\xi^{(j)}|\to\infty$, and a sequence of integral $1$-forms $\{\theta_j\}_{j\in\mathbb{N}}$ such that, 
	for every $\varepsilon>0$, the sequence
	\[
	\{\rho_j^{(\varepsilon)}\}_{j\in\mathbb{N}} = \big\{e^{\varepsilon|\xi^{(j)}|^{\frac{1}{s}}} 	(\xi^{(j)} \cdot \boldsymbol{\omega} - \theta_j)\big\}_{j\in\mathbb{N}}
	\]
	is bounded in $\mathsf{\Lambda}^1 \mathscr{G}^{(s)}(M)$.
	
	Since each $\theta_j$ is integral and belongs to 
	$\mathsf{\Lambda}^1\mathscr{G}^{(s)}(M)$, 
	we may write
	\begin{equation}\label{theta_bss}
		\theta_j = \sum_{\ell=1}^{d} \mu_{j\ell}\,\vartheta_\ell + \mathrm{d}w_j,
	\end{equation}
	where $\mu_{j\ell}\in\mathbb{Z}$ and $w_j\in\mathscr{G}^{(s)}(M)$.
	
	Fix $\varepsilon>0$. 
	Combining \eqref{omega} and \eqref{theta_bss}, we obtain
	\begin{align*}
		\rho_j^{(\varepsilon)}
		&= e^{\varepsilon|\xi^{(j)}|^{\frac{1}{s}}} \left[ 	\sum_{\ell=1}^{d} \left( \sum_{k=1}^{m} A(\boldsymbol{\omega})_{\ell k}\xi^{(j)}_k \right)\vartheta_\ell
		+ \sum_{k=1}^{m}\xi^{(j)}_k\,\mathrm{d}v_k 	- \sum_{\ell=1}^{d} \mu_{j\ell}\,\vartheta_\ell - \mathrm{d}w_j
		\right] \\ 
		&= e^{\varepsilon|\xi^{(j)}|^{\frac{1}{s}}} \sum_{\ell=1}^{d} \left( \sum_{k=1}^{m}
		A(\boldsymbol{\omega})_{\ell k}\xi^{(j)}_k - \mu_{j\ell} 	\right)\vartheta_\ell
		+ e^{\varepsilon|\xi^{(j)}|^{\frac{1}{s}}} 	\left( 	\sum_{k=1}^{m} \xi^{(j)}_k\,\mathrm{d}v_k - \mathrm{d}w_j \right).
	\end{align*}
	
	Integrating over the cycles $\sigma_\ell$,
	we obtain
	\[
	\frac{1}{2\pi} \int_{\sigma_\ell} \rho_j^{(\varepsilon)} =  	e^{\varepsilon |\xi^{(j)}|^{\frac{1}{s}}}\left( \sum_{k=1}^{m} A(\boldsymbol{\omega})_{\ell k}\xi^{(j)}_k
	- \mu_{j\ell} 	\right), 	\qquad 	\ell=1,\dots,d.
	\]
	
	Since $\{\rho_j^{(\varepsilon)}\}_{j\in\mathbb{N}}$ is bounded in  $\mathsf{\Lambda}^1 \mathscr{G}^{(s)}(M)$, it is bounded in each seminorm defining the Beurling topology. In particular, there exists a constant 
	$C_\varepsilon>0$ such that
	\[
	\left| 	\int_{\sigma_\ell} 	\rho_j^{(\varepsilon)} 	\right|
	\le 	C_\varepsilon, 	\qquad 	\forall\, j\in\mathbb{N},\  \ell=1,\dots,n.
	\]
	
	Therefore,
	\[
	e^{\varepsilon|\xi^{(j)}|^{\frac{1}{s}}} \left| \sum_{k=1}^{m} 	A(\boldsymbol{\omega})_{\ell k}\xi^{(j)}_k - \mu_{j\ell} 	\right|
	\le C_\varepsilon, 	\qquad	\forall\,j\in\mathbb{N},\ \ell=1,\dots,d.
	\]
	
	Defining $\eta^{(j)} := -(\mu_{j1},\dots,\mu_{jd})\in\mathbb{Z}^d$, 
	we obtain
	\begin{equation}\label{nDCs}
	|\eta^{(j)} + A(\boldsymbol{\omega})\xi^{(j)}| \le C_\varepsilon' e^{-\varepsilon| \xi^{(j)}|^{\frac{1}{s}}},\qquad\forall\,j\in\mathbb{N}.
	\end{equation}
	
	Since \eqref{nDCs} holds for every $\varepsilon>0$, $A(\boldsymbol{\omega})$ does not satisfy \eqref{DCss}.
	\end{proof}

	\section{Global Gevrey hypoellipticity of $\mathbb{L}$}\label{sec-sHypoel}
	
	In this section, we provide a characterization of the global hypoellipticity of $\mathbb{L}$ in Gevrey spaces of Roumieu and Beurling type for $s>1$. 
	The main tool in the proof is the theory of partial Fourier series in Gevrey classes. 
	For the convenience of the reader, the necessary definitions, results and proofs concerning this topic are collected in Appendix~\ref{appendix}.
	
	\begin{proposition}\label{uni_cov}
		Let $\Pi:\hat M\to M$ be the universal covering of $M$, and let $\omega$ be a real-valued closed $1$-form on $M$. 
		The following statements are equivalent:
		\begin{enumerate}
			\item $\omega$ is integral;
			\item For every $\psi\in C^\infty(\hat{M})$ satisfying $\mathrm{d}\psi = \Pi^*\omega$, and for all $P,Q\in\hat{M}$ with $\Pi(P)=\Pi(Q)$, one has
			\[
			\psi(P)-\psi(Q)\in 2\pi\mathbb{Z}.
			\]
		\end{enumerate}
		In particular, if $\omega$ is integral and $\psi\in C^\infty(\hat M)$ satisfies $\mathrm{d}\psi=\Pi^*\omega$, then $e^{i\psi(P)} = e^{i\psi(Q)}$ whenever $\Pi(P)=\Pi(Q)$, and hence $e^{i\psi}$ descends to a smooth function on $M$. Moreover,
		\[
		\mathrm{d}(e^{i\psi}) = i\omega\, e^{i\psi}.
		\]
	\end{proposition}

	\begin{remark}\label{hatM_analytic}
		If $M$ is an analytic manifold, then $\hat M$ is also an analytic manifold and $\Pi:\hat M\to M$ is an analytic map. Since the pullback $\Pi^*$ is a local analytic diffeomorphism, if $\omega$ is a Gevrey integral 1-form, $\Pi^{*}\omega$ is also Gevrey of same order. Moreover, the ${[s]}$-scattering metric on $M$ induces a Gevrey metric on $\hat M$ via pullback.  Hence, we can defined the transposed of an operator on $\hat M$. In particular, we can use \cite[Corollary 2.4]{CKMT} (as in Proposition \ref{gev_gh}) to obtain that the exterior derivative $\mathrm{d}$ acting on functions on $\hat M$ is $[s]$-globally hypoelliptic. Since $\psi\in C^\infty(\hat M)$ is such that $\mathrm{d}\psi=\Pi^*\omega$ we obtain $\psi\in\mathscr{G}^{[s]}(\hat M)$. In particular, if $\omega$ is integral, Proposition \ref{uni_cov} yields $e^{i\psi}\in\mathscr{G}^{[s]}(M)$.
	\end{remark}

	\begin{lemma}\label{lemma_ADL}
		Let $U\subset\mathbb{R}^n$ be an open set and let 
		$\boldsymbol{\phi}=(\phi_1,\dots,\phi_m):U\to\mathbb{R}^m$ 
		be a smooth map. 
		Assume that, for some compact set $K\Subset U$, there exist constants $C,h>0$ such that
		\[
		\sup_{t\in K} 	|\partial_t^\alpha \phi_k(t)| \le C\, h^{|\alpha|}\alpha!^s,
		\qquad 	\forall\, \alpha\in\mathbb{N}_0^n, \ k=1,\dots,m.
		\]
		Then, for every $\varepsilon>0$, there exists a constant $h_0>0$, depending only on $C,h,m$ and $\varepsilon$, such that
		\[
		\sup_{t\in K} |\partial_t^\alpha e^{i\xi\cdot\boldsymbol{\phi}(t)}| \le
		h_0^{|\alpha|}\alpha!^s \, e^{\varepsilon|\xi|^{\frac{1}{s}}}, \qquad
		\forall\,\xi\in\mathbb{Z}^m,\ \forall\,\alpha\in\mathbb{N}_0^n.
		\]
	\end{lemma}
	
	Lemma~\ref{lemma_ADL} appears in \cite[Lemma~6.1]{ADL2023gh}, and its proof follows the same pattern as the proof of \cite[Lemma~4.3]{DM2020}. 
	By adapting the argument and modifying the quantifiers accordingly, we obtain the following Beurling version.

	\begin{lemma}\label{lemma_ADL_beurling}
		Let $U\subset\mathbb{R}^n$ be an open set and let 
		$\boldsymbol{\phi}=(\phi_1,\dots,\phi_m):U\to\mathbb{R}^m$ 
		be a smooth map. 
		Assume that for every compact set $K\Subset U$ and every $h>0$, 
		there exists $C>0$ such that
		\[
		\sup_{t\in K} |\partial_t^\alpha \phi_k(t)| \le C\, h^{|\alpha|}\alpha!^s, \qquad \forall\, \alpha\in\mathbb{N}_0^n,\  k=1,\dots,m.
		\]
		
		Then, for every $h>0$, there exists $\varepsilon_0>0$ such that
		\[
		\sup_{t\in K} |\partial_t^\alpha e^{i\xi\cdot\boldsymbol{\phi}(t)}| 
		\le h^{|\alpha|}\alpha!^s \, e^{\varepsilon_0|\xi|^{\frac{1}{s}}},
		\qquad \forall\,\xi\in\mathbb{Z}^m,\ \forall\,\alpha\in\mathbb{N}_0^n.
		\]
	\end{lemma}
	\begin{proof}
		Fix a compact set $K\Subset U$ and let $h>0$ be given. By hypothesis, for every $\tilde h>0$ there exists a constant $C=C(\tilde h)>0$ such that
		\[
		\sup_{t\in K} |\partial_t^\alpha \phi_k(t)| \le C\, \tilde h^{|\alpha|}\alpha!^s,
		\qquad 	\forall\, \alpha\in\mathbb{N}_0^n, \ k=1,\dots,m.
		\]
		
		From the proof of \cite[Lemma~4.3]{DM2020}, under the above estimate one obtains
		\begin{equation}\label{est_intermediate}
			\sup_{t\in K} |\partial_t^\alpha e^{i\xi\cdot\boldsymbol{\phi}(t)}| 
			\le \big[2\tilde h (S+1)\big]^{|\alpha|} \alpha!^s 	\,  e^{\varepsilon_0|\xi|^{\frac{1}{s}}},
		\end{equation}
		for all $\alpha\in\mathbb{N}_0^n$,  where $S = n C s^s m^s \varepsilon_0^{-s}.$
		
		We now choose the parameters appropriately. 
		Let
		\[
		\tilde h = \frac{h}{4 n s^s m^s}.
		\]
		With this choice,
		\[
		2\tilde h (S+1) = 2\tilde h \big(n C s^s m^s \varepsilon_0^{-s} + 1\big).
		\]
		
		If we choose $\varepsilon_0>0$ sufficiently large so that $C \varepsilon_0^{-s} \le 1,$ then $S = n C s^s m^s \varepsilon_0^{-s} 	\le n s^s m^s.$	Hence,
		\[
		S+1 \le n s^s m^s + 1 \le 2 n s^s m^s, 
		\]
		and therefore
		\[
		2\tilde h (S+1) \le 2\tilde h (2 n s^s m^s) = h.
		\]
		
		Substituting this estimate into \eqref{est_intermediate},  we obtain
		\[
		\sup_{t\in K} |\partial_t^\alpha e^{i\xi\cdot\boldsymbol{\phi}(t)}|
		\le h^{|\alpha|} \alpha!^s \, e^{\varepsilon_0|\xi|^{\frac{1}{s}}},
		\]
		for all $\alpha\in\mathbb{N}_0^n$ and $\xi\in\mathbb{Z}^m$, which concludes the proof.
	\end{proof}
	
	\begin{proposition}\label{prop_rat}
		Fix $s>1$ and let $\boldsymbol{\omega}=(\omega_1,\dots,\omega_m)$ be a family of real-valued closed $1$-forms in $\mathsf{\Lambda}^1\mathscr{G}^{[s]}_{\partial M}(M)$. 
		If $\boldsymbol{\omega}$ is rational, then $\mathbb{L}$ is not ${[s]}$-globally hypoelliptic.
	\end{proposition}
	
	\begin{proof}
		Suppose that $\boldsymbol{\omega}$ is rational. 
		Then, as in the smooth case, there exists $\xi_0\in\mathbb{Z}^m\setminus\{0\}$ such that $\xi_0\cdot\boldsymbol{\omega}$ is an integral $1$-form.
		
		Let $\psi\in C^\infty(\hat M)$ be a function satisfying
		\[
		\mathrm{d}\psi = \Pi^*(\xi_0\cdot\boldsymbol{\omega}),
		\]
		where $\Pi:\hat M\to M$ denotes the universal covering of $M$. 
		By Proposition~\ref{uni_cov}, the function $e^{i\psi}$ descends to a smooth function on $M$.
		
		Define
		\[
		u(t,x) 	= \sum_{j=1}^\infty e^{ij\psi(t)}\,e^{-ij\xi_0\cdot x}.
		\]
		
		Then $u\in\mathscr{D}'(M\times\mathbb{T}^m)$, but $u\notin C^\infty(M\times\mathbb{T}^m)$, since its Fourier coefficients do not decay. In particular, $u\notin \mathscr{G}^{[s]} (M\times\mathbb{T}^m)$.
		
		Observe that
		\[
		\mathbb{L}\big(e^{ij\psi(t)}e^{-ij\xi_0\cdot x}\big) = e^{-ij\xi_0\cdot x}
		\big( ij\,\mathrm{d}\psi - ij\,(\xi_0\cdot\boldsymbol{\omega}) \big) e^{ij\psi(t)} =0,
		\]

		Thus $\mathbb{L}u=0\in\mathscr{G}^{[s]}(M\times\mathbb{T}^m)$, while 
		$u\notin\mathscr{G}^{[s]}(M\times\mathbb{T}^m)$. 
		Therefore, $\mathbb{L}$ is not ${[s]}$-globally hypoelliptic.
	\end{proof}

\begin{theorem}\label{gh_roumieu}
	Fix $s>1$ and let $\boldsymbol{\omega}=(\omega_1,\dots,\omega_m)$ be a family of real-valued closed 1-forms on $\mathsf{\Lambda}^1\mathscr{G}^s_{\partial M}(M)$. The operator $\mathbb{L}$ is $s$-globally hypoelliptic 
	if and only if $\boldsymbol{\omega}$ is neither rational nor $s$-exponential Liouville.
\end{theorem}
\begin{proof}
	
	If $\boldsymbol{\omega}$ is rational, then $\mathbb{L}$ is not $s$-globally hypoelliptic by Proposition \ref{prop_rat}. 
	If $\boldsymbol{\omega}$ is $s$-exponential Liouville, then there exist $\varepsilon>0$, a sequence $\{\xi^{(j)}\}_{j\in\mathbb{N}}$ in $\mathbb{Z}^m$ such that $|\xi^{(j)}|\to\infty$, and a sequence $\{\theta_j\}_{j\in\mathbb{N}}$ of integral $1$-forms in $\mathsf{\Lambda}^1\mathscr{G}^s(M)$ such that 
	\[
	\{e^{\varepsilon|\xi^{(j)}|^{\frac{1}{s}}}(\xi^{(j)}\cdot\boldsymbol{\omega}-\theta_j)\}
	\]
	is bounded in $\mathsf{\Lambda}^1\mathscr{G}^s(M)$.
	
	For each $j\in\mathbb{N}$, let $\psi_j\in C^\infty(\hat M)$ be such that $\mathrm{d}\psi_j=\Pi^*\theta_j$. 
	By Remark \ref{hatM_analytic}, each $e^{i\psi_j}$ belongs to $\mathscr{G}^s(M)$. 
	Define
	\[
	u = \sum_{j=1}^{\infty} e^{i\psi_j}e^{-i\xi^{(j)}\cdot x}.
	\]
	
	It is clear that $u\in \mathscr{D}'(M\times\mathbb{T}^m)\setminus\mathscr{G}^s(M\times\mathbb{T}^m)$. Moreover,
	\[
	f=\mathbb{L}u 
	= -i\sum_{j=1}^\infty (\xi^{(j)}\cdot\boldsymbol{\omega}-\theta_j)e^{i\psi_j}e^{-i\xi^{(j)}\cdot x}.
	\]
	
	Let
	\[
	\rho_j = e^{\varepsilon|\xi^{(j)}|^{\frac{1}{s}}}(\xi^{(j)}\cdot\boldsymbol{\omega}-\theta_j).
	\]
	
	Then
	\[
	\widehat{f}_\xi = 
	\begin{cases}
		-ie^{-\varepsilon|\xi^{(j)}|^{\frac{1}{s}}}e^{i\psi_j}\rho_j,&\text{if } \xi=\xi^{(j)},\\
		0,&\text{otherwise}.
	\end{cases}
	\]
	
	Let $(V,t)$ be a local chart such that $\Pi:\hat V\to V$ is an analytic diffeomorphism for some open subset $\hat V\subset\hat M$, and let $K\Subset V$. 
	Since $\{\rho_j\}_{j\in\mathbb{N}}$ is bounded in $\mathsf{\Lambda}^1\mathscr{G}^s(M)$, there exist $C_1,h_1>0$ such that
	\[
	\sup_{t\in K}|\partial_t^\alpha\rho_j(t)|
	\leq C_1h_1^{|\alpha|}\alpha!^s,
	\qquad\forall\,\alpha\in\mathbb{N}_0^n.
	\]

	Also, we have on $V$ that
	\[\theta_j = \mathrm{d}(\psi_j\circ\Pi^{-1}) = \sum_{\nu=1}^{n}\partial_{t_\nu}(\psi_j\circ\Pi^{-1})\mathrm{d}t_\nu.\]
	
	The boundedness of $\{\rho_j\}_{j\in\mathbb{N}}$ also implies the boundedness of the sequence $\{|\xi^{(j)}|^{-1}\theta_j\}_{j\in\mathbb{N}}$. 
	Hence, by Lemma \ref{lemma_ADL}, there exist $C_2,h_2>0$ such that
	\[
	\sup_{t\in K}|\partial_t^\alpha e^{i\psi_j\circ\Pi^{-1}(t)}| = \sup_{t\in K}|\partial_t^\alpha e^{i|\xi^{(j)}|(|\xi^{(j)}|^{-1}\psi_j\circ\Pi^{-1}(t))}| \leq C_2h_2^{|\alpha|}\alpha!^s e^{\frac{\varepsilon}{2}| \xi^{(j)}|^{\frac{1}{s}}},
	\]
	for all $\alpha \in \mathbb{N}_0^n.$
	
	By Leibniz rule,
	\[
	\sup_{t\in K}|\partial_t^\alpha(e^{i\psi_j\circ\Pi^{-1}}\rho_j)|
	\leq C_3h_3^{|\alpha|}\alpha!^s e^{\frac{\varepsilon}{2}|\xi^{(j)}|^{\frac{1}{s}}},
	\]
	for some $C_3,h_3>0$. Therefore,
	\[
	\sup_{t\in K}|\partial_t^{\alpha}\widehat{f}_\xi|
	\leq C_3h_3^{|\alpha|}\alpha!^s e^{-\frac{\varepsilon}{2}|\xi^{(j)}|^{\frac{1}{s}}},
	\]
	which implies that $f\in\mathsf{\Lambda}^1\mathscr{G}^s(M\times\mathbb{T}^m)$. 
	Hence $\mathbb{L}$ is not $s$-globally hypoelliptic.
	
	\medskip

	Now suppose that $\boldsymbol{\omega}$ is neither rational nor $s$-exponential Liouville and let 
	$u\in\mathscr{D}'(M\times\mathbb{T}^m)$ satisfy 
	$\mathbb{L}u=f\in\mathsf{\Lambda}^1\mathscr{G}^s(M\times\mathbb{T}^m)$. 
	Then, for each $\xi\in\mathbb{Z}^m$,
	\[
	\widehat{f}_\xi = \mathrm{d}\widehat{u}_\xi + i(\xi\cdot \boldsymbol{\omega}) \widehat{u}_\xi.
	\]
	
	For each $k$, let $\psi_k\in C^\infty(\hat M)$ be such that 
	$\mathrm{d}\psi_k = \Pi^*\omega_k$, and write 
	$\boldsymbol{\psi}=(\psi_1,\dots,\psi_m)$. 
	Then
	\begin{equation}\label{deuef_psi}
	\mathrm{d}(e^{i\xi\cdot\boldsymbol{\psi}}\Pi^*\widehat{u}_\xi)
	=	e^{i\xi\cdot\boldsymbol{\psi}}\Pi^*\widehat{f}_\xi.
	\end{equation}

	Fix $t_0\in M$ and choose a coordinate neighbourhood $B$ of $t_0$
	diffeomorphic to an open ball, small enough so that
	$\Pi:\hat B\to B$ is an analytic diffeomorphism
	for some open set $\hat B\subset\hat M$.
	
	Let $\boldsymbol{\phi} 	= (\psi_1\circ\Pi^{-1},\dots,\psi_m\circ\Pi^{-1})$ on $B$. Then \eqref{deuef_psi} becomes, on $B$,
	\begin{equation}\label{deuef_local}
		\mathrm{d}	\big( e^{i\xi\cdot\boldsymbol{\phi}} \widehat{u}_\xi
		\big) = e^{i\xi\cdot\boldsymbol{\phi}} 	\widehat{f}_\xi.
	\end{equation}
	
	Given \(t\in B\), let \([t_0,t]\) be a smooth path in \(B\) joining \(t_0\) to \(t\).
	Integrating \eqref{deuef_local} along \([t_0,t]\), we obtain
	\begin{equation*}
		\widehat{u}_\xi(t)
		=
		e^{i\xi\cdot (\boldsymbol{\phi}(t_0)-\boldsymbol{\phi}(t))}\,\widehat{u}_\xi(t_0)
		+
		e^{-i\xi\cdot\boldsymbol{\phi}(t_0)}
		\int_{t_0}^{t} e^{i\xi\cdot\boldsymbol{\phi}}\,\widehat{f}_\xi .
	\end{equation*}
	
	We first derive estimates for the coefficients \(\widehat{u}_\xi(t_0)\).
	By the Hurewicz Theorem, we may choose the generators \(\sigma_\ell\) to be smooth loops based at \(t_0\).
	Fix a point \(Q_0\in \Pi^{-1}(t_0)\subset \hat M\).
	For each \(\ell=1,\dots,d\), let \(\hat\sigma_\ell:[0,2\pi]\to \hat M\) be the lift of \(\sigma_\ell\) satisfying \(\hat\sigma_\ell(0)=Q_0\),
	and set \(Q_\ell:=\hat\sigma_\ell(2\pi)\).
	Using \eqref{deuef_psi} and Stokes' theorem, we have
	\begin{equation*}
		\int_{\hat\sigma_\ell} e^{i\xi\cdot\boldsymbol{\psi}}\,\Pi^*\widehat{f}_\xi
		= \int_{\hat\sigma_\ell} \mathrm{d}\!\left(e^{i\xi\cdot\boldsymbol{\psi}}\,\widehat{u}_\xi\right)
		= \left(e^{i\xi\cdot\boldsymbol{\psi}(Q_\ell)} - e^{i\xi\cdot\boldsymbol{\psi}(Q_0)}\right)\widehat{u}_\xi(t_0),
	\end{equation*}
	and therefore
	\begin{equation}\label{u_t0_loop}
		\bigl(1-e^{i\xi\cdot(\boldsymbol{\psi}(Q_0)-\boldsymbol{\psi}(Q_\ell))}\bigr)\,\widehat{u}_\xi(t_0)
		= e^{-i\xi\cdot\boldsymbol{\psi}(Q_\ell)} \int_{\hat\sigma_\ell} e^{i\xi\cdot\boldsymbol{\psi}}\,\Pi^*\widehat{f}_\xi .
	\end{equation}
	
	Next, we compute the increments of \(\boldsymbol{\psi}\) along the lifts:
	for each \(k=1,\dots,m\),
	\[
	\psi_k(Q_\ell)-\psi_k(Q_0) = \int_{\hat\sigma_\ell}\mathrm{d}\psi_k
	= \int_{\sigma_\ell}\omega_k
	= 2\pi\,A(\boldsymbol{\omega})_{\ell k}.
	\]
	Consequently, for \(\xi\in\mathbb{Z}^m\),
	\[
	A(\boldsymbol{\omega})\xi
	= \frac1{2\pi}\Bigl( \xi\cdot\bigl[\boldsymbol{\psi}(Q_1)-\boldsymbol{\psi}(Q_0)\bigr],
	\,\dots,\, \xi\cdot\bigl[\boldsymbol{\psi}(Q_d)-\boldsymbol{\psi}(Q_0)\bigr]
	\Bigr).
	\]
	
	Since \(\boldsymbol{\omega}\) is not \(s\)-exponential Liouville, condition \eqref{DCs} fails by Proposition~\ref{liou_DCs}.
	Hence, by Lemma~\ref{est_DCs}, for every \(\varepsilon>0\) there exist \(C_\varepsilon>0\) and some \(\ell\in\{1,\dots,d\}\) such that
	\begin{equation}\label{est_exp}
		\bigl|1-e^{i\xi\cdot(\boldsymbol{\psi}(Q_0)-\boldsymbol{\psi}(Q_\ell))}\bigr|
		\geq C_\varepsilon\,e^{-\varepsilon|\xi|^{\frac1s}},
		\qquad \forall\,\xi\in\mathbb{Z}^m\setminus\{0\}.
	\end{equation}
	
	In particular, the left-hand side of \eqref{est_exp} does not vanish.
	Hence, by \eqref{u_t0_loop} and \eqref{est_exp},
	\begin{equation*}
		|\widehat{u}_\xi(t_0)|
		=
		\dfrac{|e^{-i\xi\cdot\boldsymbol{\psi}(Q_\ell)}|}{|1-e^{i\xi\cdot(\boldsymbol{\psi}(Q_0)-\boldsymbol{\psi}(Q_\ell))}|}
		\left|\int_{\hat\sigma_\ell}e^{i\xi\cdot\boldsymbol{\psi}}\Pi^*\widehat{f}_\xi\right|
		\leq
		C_\varepsilon^{-1}e^{\varepsilon|\xi|^{\frac{1}{s}}}
		\left|\int_{\hat\sigma_\ell}e^{i\xi\cdot\boldsymbol{\psi}}\Pi^*\widehat{f}_\xi\right|.
	\end{equation*}
	
	Since \(\sigma_\ell\) is compact, let \(\{U_i\}_{i\in I}\) be a finite family of coordinate domains that cover \(\sigma_\ell\) and,
	for each \(i\in I\), take \(K_i\Subset U_i\) whose interiors still cover \(\sigma_\ell\).
	Also, consider a partition
	\[
	0=\tau_0<\tau_1<\cdots<\tau_L=2\pi
	\]
	such that, for each \(r=1,\dots,L\), the segment \(\sigma_\ell([\tau_{r-1},\tau_r])\) is entirely contained in the interior of some \(K_i\),
	with \(i\) depending on \(r\).
	If \((t_1,\dots,t_n)\) are local coordinates on \(K_i\), we have on \([\tau_{r-1},\tau_r]\subset\mathbb{R}\) that
	\[
	\sigma_\ell^*\widehat{f}_\xi
	= \sum_{\nu=1}^{n}((\widehat{f_\nu})_\xi\circ\sigma_\ell)\,\mathrm{d}(t_\nu\circ\sigma_\ell)
	= \sum_{\nu=1}^{n}((\widehat{f_\nu})_\xi\circ\sigma_\ell)\,G_\nu\,\mathrm{d}\tau,
	\]
	for some continuous functions \(G_\nu\) on \([\tau_{r-1},\tau_r]\), $\nu=1,\dots,n$. Hence,
	\[
	\int_{\tau_{r-1}}^{\tau_r} e^{i\xi\cdot\boldsymbol{\psi}\circ\hat\sigma_\ell}
	\,\sigma_\ell^*\widehat{f}_\xi 
	= \sum_{\nu=1}^{n}\int_{\tau_{r-1}}^{\tau_r} e^{i\xi\cdot \boldsymbol{\psi} \circ\hat\sigma_\ell} 	(\widehat{f_\nu})_\xi (\sigma_\ell(\tau)) \,G_\nu(\tau) \,\mathrm{d}\tau.
	\]

	Since \(f\in\mathsf{\Lambda}^1\mathscr{G}^s(M\times\mathbb{T}^m)\), there exist \(C,\delta>0\) independent of \(r\) such that
	\[
	\sup_{1\leq \nu\leq n}\sup_{t\in K_i}|(\widehat{f_\nu})_\xi(t)| \leq Ce^{-\delta|\xi|^{\frac{1}{s}}}.
	\]
	
	Then, we have
	\[
	\left|\int_{\tau_{r-1}}^{\tau_r}
	e^{i\xi\cdot\boldsymbol{\psi}\circ\hat\sigma_\ell}\,\sigma_\ell^*\widehat{f}_\xi\right| \leq \sum_{\nu=1}^n\int_{\tau_{r-1}}^{\tau_r} |(\widehat{f_\nu})_\xi (\sigma_\ell(\tau))|\,|G_\nu(\tau)|\,\mathrm{d}\tau
	\leq
	C'e^{-\delta|\xi|^{\frac{1}{s}}},
	\]
	for some \(C'>0\) independent of \(\xi\in\mathbb{Z}^m\) and \(r\in\{1,\dots,L\}\).
	Finally,
	\begin{align*}
		\left|\int_{\hat{\sigma}_\ell}e^{i\xi\cdot \boldsymbol{\psi}}\Pi^*\widehat{f}_\xi\right|
		& \leq 	\sum_{r=1}^L \left|\int_{\tau_{r-1}}^{\tau_r} e^{i\xi\cdot \boldsymbol{\psi}\circ\hat\sigma_\ell}\,\sigma_\ell^*\widehat{f}_\xi\right|
		\leq C''e^{-\delta|\xi|^{\frac{1}{s}}},
	\end{align*}
	where \(C''=C'L>0\).
	Choosing \(\varepsilon=\delta/2\), we obtain
	\begin{equation*}
		|\widehat{u}_\xi(t_0)| 	\leq C_0e^{-\frac{\delta}{2}|\xi|^{\frac{1}{s}}},
		\qquad\forall\,\xi\in\mathbb{Z}^m,
	\end{equation*}
	where \(C_0=C''C_\varepsilon^{-1}>0\).

	Now, we can estimate the derivatives of \(\widehat{u}_\xi(t)\).
	Given \(K\Subset B\), Lemma \ref{lemma_ADL} yields
	\begin{equation}\label{est_eu}
		\left|\partial_t^\alpha\!\left(e^{i\xi\cdot (\boldsymbol{\phi}(t_0)-\boldsymbol{\phi}(t))}\widehat{u}_\xi(t_0)\right)\right|
		\leq C_0e^{-\frac{\delta}{2}|\xi|^{\frac{1}{s}}}\,|\partial_t^\alpha e^{i\xi\cdot\boldsymbol{\phi}(t)}|
		\leq C_1h_1^{|\alpha|}\alpha!^s e^{-\frac{\delta}{4}|\xi|^{\frac{1}{s}}},
	\end{equation}
	for all \(\alpha\in\mathbb{N}_0^n\), \(t\in K\), and \(\xi\in\mathbb{Z}^m\), with \(C_1,h_1>0\).
	
	For the second term of \(\widehat{u}_\xi\), we can write
	\[
	e^{-i\xi\cdot\boldsymbol{\phi}(t_0)}\int_{t_0}^t e^{i\xi\cdot\boldsymbol{\phi}}\widehat{f}_\xi
	= e^{-i\xi\cdot\boldsymbol{\phi}(t_0)}F_\xi,
	\]
	where \(F_\xi\) is such that \(\mathrm{d}F_\xi = e^{i\xi\cdot\boldsymbol{\phi}}\widehat{f}_\xi\).
	
	Since \(f\in\mathsf{\Lambda}^1\mathscr{G}^s(M\times\mathbb{T}^m)\), there exist \(C_2,h_2,\delta_0>0\) such that
	\begin{equation}\label{est_f} 
		\sup_{t\in K}|\partial_t^\alpha \widehat{f}_\xi(t)|
		\leq C_2h_2^{|\alpha|}\alpha!^s e^{-\delta_0|\xi|^{\frac{1}{s}}},\qquad\forall\,\xi\in\mathbb{Z}^m,\ \forall\,\alpha\in\mathbb{N}_0^n.
	\end{equation}
	
	Equation \eqref{est_f} and Lemma \ref{lemma_ADL} imply that there exist \(h_3>0\) such that, for all \(\alpha\in\mathbb{N}_0^n\) and \(\xi\in\mathbb{Z}^m\),
	\begin{align*}
		|\partial_t^\alpha \mathrm{d}F_\xi|
		&\leq \sum_{\beta\leq\alpha}\binom{\alpha}{\beta}
		|\partial_t^{\alpha-\beta}e^{i\xi\cdot\boldsymbol{\phi}}|\,
		|\partial_t^{\beta}\widehat{f}_\xi|\\
		&\leq
		\sum_{\beta\leq\alpha}\binom{\alpha}{\beta}
		h_3^{|\alpha|-|\beta|}(\alpha-\beta)!^s e^{\frac{\delta_0}{2}|\xi|^{\frac{1}{s}}}
		\, C_2h_2^{|\beta|}\beta!^s e^{-\delta_0|\xi|^{\frac{1}{s}}}\\
		&\leq C_3h_4^{|\alpha|}\alpha!^s e^{-\frac{\delta_0}{2}|\xi|^{\frac{1}{s}}},
	\end{align*}
	for certain \(C_3,h_4>0\).
	Hence, we can find \(C_4,h_5>0\) such that
	\[
	\sup_{t\in K}|\partial_t^\alpha F_\xi(t)|
	\leq C_4h_5^{|\alpha|}\alpha!^s e^{-\frac{\delta_0}{2}|\xi|^{\frac{1}{s}}},\qquad\forall\,\xi\in\mathbb{Z}^m,\ \forall\,\alpha\in\mathbb{N}_0^n.
	\]
	
	Hence, we can apply the same ideas to obtain \(C_5,h_6>0\) such that
	\begin{equation}\label{est_eF}
		\sup_{t\in K}\bigl|\partial_t^\alpha(e^{-i\xi\cdot\boldsymbol{\phi}(t)}F_\xi(t))\bigr|
		\leq
		C_5h_6^{|\alpha|}\alpha!^s e^{-\frac{\delta_0}{2}|\xi|^{\frac{1}{s}}},\qquad\forall\,\xi\in\mathbb{Z}^m,\ \forall\,\alpha\in\mathbb{N}_0^n.
	\end{equation}
	
	Hence, combining \eqref{est_eu} and \eqref{est_eF}, we obtain \(C,h,\varepsilon>0\) such that
	\[
	\sup_{t\in K}|\partial_t^\alpha\widehat{u}_\xi(t)|
	\leq Ch^{|\alpha|}\alpha!^s e^{-\varepsilon|\xi|^{\frac{1}{s}}},\qquad\forall\,\xi\in\mathbb{Z}^m,\ \forall\,\alpha\in\mathbb{N}_0^n.
	\]
	
	Now, given any \(K\Subset M\), we can cover \(K\) with finitely many open coordinate patches \(B\) such that
	\(\Pi:\hat{B}\to B\) is a diffeomorphism, so that the previous estimates still hold on each $B$.
	This shows that \(u\in\mathscr{G}^s(M\times\mathbb{T}^m)\), which implies the \(s\)-global hypoellipticity of \(\mathbb{L}\).
\end{proof}

	The proof of the previous theorem follows the same strategy as in \cite[Theorem 3.4]{ADL2023gh}. By adapting those arguments, we also obtain a characterization of the global hypoellipticity of \(\mathbb{L}\) on Beurling-type Gevrey spaces:
	
	\begin{theorem}\label{gh_beurling}
		Fix \(s>1\) and let \(\boldsymbol{\omega}=(\omega_1,\dots,\omega_m)\) be a family of real-valued closed \(1\)-forms in \(\mathsf{\Lambda}^1\mathscr{G}^{(s)}_{\partial M}(M)\).
		The operator \(\mathbb{L}\) is \((s)\)-globally hypoelliptic if and only if the family \(\boldsymbol{\omega}\) is neither rational nor \((s)\)-exponential Liouville.
	\end{theorem}
	
	\begin{proof}
		If \(\boldsymbol{\omega}\) is rational, then \(\mathbb{L}\) is not \((s)\)-globally hypoelliptic by Proposition~\ref{prop_rat}.
		Now assume that \(\boldsymbol{\omega}\) is \((s)\)-exponential Liouville.
		Then there exist a sequence \(\{\xi^{(j)}\}_{j\in\mathbb{N}}\) in \(\mathbb{Z}^m\) with \(|\xi^{(j)}|\to\infty\) and a sequence \(\{\theta_j\}_{j\in\mathbb{N}}\) of integral \(1\)-forms in \(\mathsf{\Lambda}^1\mathscr{G}^{(s)}(M)\) such that, for every \(\varepsilon>0\), the sequence
		\[
		\bigl\{e^{\varepsilon|\xi^{(j)}|^{\frac{1}{s}}}\,(\xi^{(j)}\cdot\boldsymbol{\omega}-\theta_j)\bigr\}_{j\in\mathbb{N}}
		\]
		is bounded in \(\mathsf{\Lambda}^1\mathscr{G}^{(s)}(M)\).
		
		For each \(j\in\mathbb{N}\), let \(\psi_j\in C^\infty(\hat M)\) satisfy \(\mathrm{d}\psi_j=\Pi^*\theta_j\).
		Moreover, each \(e^{i\psi_j}\) belongs to \(\mathscr{G}^{(s)}(M)\) by Remark~\ref{hatM_analytic}.
		Set
		\[
		u=\sum_{j=1}^{\infty} e^{i\psi_j}e^{-i\xi^{(j)}\cdot x}.
		\]
		
		As in the proof of Theorem~\ref{gh_roumieu}, we have
		\(u\in \mathscr{D}' (M\times\mathbb{T}^m) \setminus \mathscr{G}^{(s)} (M\times\mathbb{T}^m) \) and
		\[
		\mathbb{L}u=f=-i\sum_{j=1}^\infty (\xi^{(j)}\cdot\boldsymbol{\omega}-\theta_j)e^{i\psi_j}e^{-i\xi^{(j)}\cdot x}.
		\]

		Then it is enough to show that \(f\in\mathsf{\Lambda}^1\mathscr{G}^{(s)}(M\times\mathbb{T}^m)\).
		Fix \(\varepsilon>0\) and set
		\[
		\rho_j^{(\varepsilon)} = e^{\varepsilon|\xi^{(j)}|^{\frac{1}{s}}}\,(\xi^{(j)}\cdot\boldsymbol{\omega}-\theta_j).
		\]
		
		Then
		\[
		\widehat{f}_\xi =
		\begin{cases}
			-ie^{-\varepsilon|\xi^{(j)}|^{\frac{1}{s}}}e^{i\psi_j}\rho_j^{(\varepsilon)}, & \text{if } \xi=\xi^{(j)},\\
			0, & \text{otherwise}.
		\end{cases}
		\]
		
		Let \((V,t)\) be a local chart small enough so that \(\Pi:\hat V\to V\) is an analytic diffeomorphism for some open subset \(\hat V\subset\hat M\),
		and let \(K\Subset V\).
		Since \(\{\rho_j^{(\varepsilon)}\}_{j\in\mathbb{N}}\) is bounded in \(\mathsf{\Lambda}^1\mathscr{G}^{(s)}(M)\), for every \(h>0\) there exists \(C_1>0\) such that
		\begin{equation}\label{est_rhoj_b}
			\sup_{t\in K}|\partial_t^\alpha\rho_j^{(\varepsilon)}(t)|
			\leq C_1h^{|\alpha|}\alpha!^s,
			\qquad \forall\, \alpha\in\mathbb{N}_0^n.
		\end{equation}
		
		Moreover, on \(V\) we have
		\[
		\theta_j = \mathrm{d}(\psi_j\circ\Pi^{-1})
		= \sum_{\nu=1}^{n}\partial_{t_\nu}(\psi_j\circ\Pi^{-1})\,\mathrm{d}t_\nu.
		\]
		
		Notice that the boundedness of \(\{\rho_j^{(\varepsilon)}\}_{j\in\mathbb{N}}\) implies the boundedness of \(\{|\xi^{(j)}|^{-1}\theta_j\}_{j\in\mathbb{N}}\) in \(\mathsf{\Lambda}^1\mathscr{G}^{(s)}(M)\).
		Hence, applying the scalar version of Lemma~\ref{lemma_ADL_beurling}, we obtain that, for every \(h>0\), there exists \(\varepsilon_0>0\) such that
		\begin{equation}\label{est_eipsi_b}
			\sup_{t\in K}\bigl|\partial_t^\alpha e^{i\psi_j\circ\Pi^{-1}(t)}\bigr|
			=
			\sup_{t\in K}\bigl|\partial_t^\alpha e^{i|\xi^{(j)}|\,(|\xi^{(j)}|^{-1}\psi_j\circ\Pi^{-1}(t))}\bigr|
			\leq h^{|\alpha|}\alpha!^s e^{\varepsilon_0|\xi^{(j)}|^{\frac{1}{s}}},
			\qquad \forall\, \alpha\in\mathbb{N}_0^n.
		\end{equation}
		
		Now, by the Leibniz rule and the estimates \eqref{est_rhoj_b} and \eqref{est_eipsi_b}, for every \(h>0\) we can find \(C>0\) such that
		\[
		\sup_{t\in K}\bigl|\partial_t^\alpha\bigl(e^{i\psi_j\circ\Pi^{-1}}\,\rho_j^{(\varepsilon)}\bigr)(t)\bigr|
		\leq
		Ch^{|\alpha|}\alpha!^s e^{\varepsilon_0|\xi^{(j)}|^{\frac{1}{s}}},
		\qquad \forall\, \alpha\in\mathbb{N}_0^n.
		\]
		
		Given \(\varepsilon>0\), if we replace \(\varepsilon\) by \(\varepsilon+\varepsilon_0\) in the definition of \(\rho_j^{(\varepsilon)}\), the previous estimate yields
		\[
		\sup_{t\in K}\bigl|\partial_t^{\alpha}\widehat{f}_\xi(t)\bigr|
		\leq
		Ch^{|\alpha|}\alpha!^s e^{-\varepsilon|\xi^{(j)}|^{\frac{1}{s}}},
		\qquad \forall\, \alpha\in\mathbb{N}_0^n.
		\]
		
		Therefore \(f\in\mathsf{\Lambda}^1\mathscr{G}^{(s)}(M\times\mathbb{T}^m)\), and hence \(\mathbb{L}\) is not \((s)\)-globally hypoelliptic.
		
		Now, suppose that \(\boldsymbol{\omega}\) is neither rational nor \((s)\)-exponential Liouville, and let
		\(u\in\mathscr{D}'(M\times\mathbb{T}^m)\) satisfy
		\(\mathbb{L}u=f\in\mathsf{\Lambda}^{0,1}\mathscr{G}^{(s)}(M\times\mathbb{T}^m)\).
		Then, for each \(\xi\in\mathbb{Z}^m\),
		\[
		\widehat{f}_\xi = \mathrm{d}\widehat{u}_\xi+i(\xi\cdot\boldsymbol{\omega})\,\widehat{u}_\xi.
		\]
		
		For each \(k=1,\dots,m\), let \(\psi_k\in C^\infty(\hat M)\) be such that \(\mathrm{d}\psi_k = \Pi^*\omega_k\).
		Denoting \(\boldsymbol{\psi}=(\psi_1,\dots,\psi_m)\), we obtain, for every \(\xi\in\mathbb{Z}^m\),
		\begin{equation}\label{deuef_psi_b}
			\mathrm{d}\!\left(e^{i\xi\cdot\boldsymbol{\psi}}\,\Pi^*\widehat{u}_\xi\right)
			=
			e^{i\xi\cdot\boldsymbol{\psi}}\,\Pi^*\widehat{f}_\xi.
		\end{equation}
		
		Fix \(t_0\in M\) and let \((B,t)\) be a local chart at \(t_0\) such that \(B\) is diffeomorphic to an open ball and small enough so that
		\(\Pi:\hat B\to B\) is an analytic diffeomorphism for some open set \(\hat B\subset \hat M\).
		If we set \(\boldsymbol{\phi}=(\psi_1\circ\Pi^{-1},\dots,\psi_m\circ\Pi^{-1})\) on \(B\), then for each \(\xi\in\mathbb{Z}^m\) we obtain
		\begin{equation}\label{deuef_b}
			\mathrm{d}(e^{i\xi\cdot\boldsymbol{\phi}}\widehat{u}_\xi) = e^{i\xi\cdot\boldsymbol{\phi}}\widehat{f}_\xi.
		\end{equation}
		
		Given \(t\in B\), let \([t_0,t]\) be a path joining \(t_0\) to \(t\).
		Integrating \eqref{deuef_b} over \([t_0,t]\) yields
		\begin{equation}\label{uxit0_b}
			\widehat{u}_\xi(t)
			= e^{i\xi\cdot (\boldsymbol{\phi}(t_0) - \boldsymbol{\phi}(t))}\widehat{u}_\xi(t_0) + e^{-i\xi\cdot\boldsymbol{\phi} (t_0)}\int_{t_0}^t e^{i\xi\cdot\boldsymbol{\phi}}\widehat{f}_\xi.
		\end{equation}
		
		First, let us estimate the sequence \(\{\widehat{u}_\xi(t_0)\}\).
		As in the previous theorem, the Hurewicz Theorem allows us to choose the loops \(\sigma_\ell\) to be smooth and based at \(t_0\).
		Fix \(Q_0\in \Pi^{-1}(t_0)\subset \hat M\).
		For each \(\ell=1,\dots,d\), let \(\hat\sigma_\ell:[0,2\pi]\to\hat M\) be the lift of \(\sigma_\ell\) such that
		\(\hat\sigma_\ell(0)=Q_0\), and set \(Q_\ell:=\hat\sigma_\ell(2\pi)\).
		By \eqref{deuef_psi_b} and Stokes' Theorem, we obtain
		\begin{equation*}
			\int_{\hat\sigma_\ell}e^{i\xi\cdot\boldsymbol{\psi}}\,\Pi^*\widehat{f}_\xi
			= \int_{\hat\sigma_\ell} \mathrm{d}\!\left(e^{i\xi\cdot\boldsymbol{\psi}} \,\Pi^*\widehat{u}_\xi\right)
			= \left(e^{i\xi\cdot\boldsymbol{\psi} (Q_\ell)}-e^{i\xi\cdot\boldsymbol{ \psi}(Q_0)}\right)\widehat{u}_\xi(t_0),
		\end{equation*}
		which yields
		\begin{equation*}
			\widehat{u}_\xi(t_0)
			= e^{i\xi\cdot (\boldsymbol{\psi}(Q_0) - \boldsymbol{\psi}(Q_\ell))} \,\widehat{u}_\xi(t_0) + e^{-i\xi\cdot\boldsymbol{\psi}(Q_\ell)}
			\int_{\hat\sigma_\ell}e^{i\xi\cdot\boldsymbol{\psi}}\,\Pi^*\widehat{f}_\xi.
		\end{equation*}
		
		Now, since
		\[
		\psi_k(Q_\ell)-\psi_k(Q_0) = \int_{\hat\sigma_\ell}\mathrm{d}\psi_k
		= \int_{\sigma_\ell}\omega_k = 2\pi A(\boldsymbol{\omega})_{\ell k},
		\]
		we obtain
		\[
		A(\boldsymbol{\omega})\xi
		= \dfrac{1}{2\pi}\left( \xi\cdot [\boldsymbol{\psi}(Q_1) - \boldsymbol{\psi}(Q_0)],\cdots, \xi\cdot  [\boldsymbol{\psi}(Q_d) - \boldsymbol{\psi}(Q_0)] \right),  \qquad \xi\in\mathbb{Z}^m.
		\]
		
		Since \(\boldsymbol{\omega}\) is not \((s)\)-exponential Liouville, \eqref{DCss} does not hold by Proposition~\ref{liou_DCs}.
		Hence, by Lemma~\ref{est_DCss}, there exist \(\varepsilon_0, C>0\) and \(\ell\in\{1,\dots,d\}\) such that
		\begin{equation}\label{est_exp_b}
			|1-e^{i\xi\cdot(\boldsymbol{\psi}(Q_0)-\boldsymbol{\psi}(Q_\ell))}|
			\geq
			C e^{-\varepsilon_0|\xi|^{\frac{1}{s}}},
			\qquad\forall\,\xi\in\mathbb{Z}^m\setminus\{0\}.
		\end{equation}
		
		Therefore, using \eqref{est_exp_b} and the identity obtained above for \(\widehat{u}_\xi(t_0)\), we have
		\begin{equation*}
			|\widehat{u}_\xi(t_0)|
			=
			\dfrac{|e^{-i\xi\cdot\boldsymbol{\psi}(Q_\ell)}|}
			{|1-e^{i\xi\cdot(\boldsymbol{\psi}(Q_0)-\boldsymbol{\psi}(Q_\ell))}|}
			\left|\int_{\hat\sigma_\ell}e^{i\xi\cdot\boldsymbol{\psi}}\Pi^*\widehat{f}_\xi\right|
			\leq
			C^{-1}e^{\varepsilon_0|\xi|^{\frac{1}{s}}}
			\left|\int_{\hat\sigma_\ell}e^{i\xi\cdot\boldsymbol{\psi}}\Pi^*\widehat{f}_\xi\right|,
		\end{equation*}
		since the left-hand side of \eqref{est_exp_b} does not vanish.
		
		By the compactness of \(\sigma_\ell\), let \(\{U_i\}_{i\in I}\) be a finite family of coordinate domains that cover \(\sigma_\ell\) and, for each \(i\in I\), choose \(K_i\Subset U_i\) such that the interiors of the \(K_i\) still cover \(\sigma_\ell\).
		Consider also a partition
		\[
		0=\tau_0<\tau_1<\cdots<\tau_L=2\pi
		\]
		of the interval \([0,2\pi]\) such that, for each \(r=1,\dots,L\), the segment \(\sigma_\ell([\tau_{r-1},\tau_r])\) is entirely contained in the interior of some \(K_i\), with \(i\) depending on \(r\).
		If \((t_1,\dots,t_n)\) are local coordinates on \(K_i\), then on \([\tau_{r-1},\tau_r]\subset\mathbb{R}\) we have
		\[
		\sigma_\ell^*\widehat{f}_\xi = \sum_{\nu=1}^{n}((\widehat{f_\nu})_\xi\circ\sigma_\ell) \, \mathrm{d}(t_\nu\circ\sigma_\ell)
		= \sum_{\nu=1}^{n}((\widehat{f_\nu})_\xi\circ\sigma_\ell)\,G_\nu\,\mathrm{d}\tau,
		\]
		for some continuous functions \(G_\nu\) on \([\tau_{r-1},\tau_r]\), $\nu=1,\dots,n$.
		Hence,
		\[
		\int_{\tau_{r-1}}^{\tau_r} 	\textbf{\(e^{i\xi\cdot\boldsymbol{\psi} \circ\hat\sigma_\ell}\)} \,\sigma_\ell^*\widehat{f}_\xi
		= \sum_{\nu=1}^{n}\int_{\tau_{r-1}}^{\tau_r} 		e^{i\xi\cdot\boldsymbol{\psi} \circ\hat\sigma_\ell} (\widehat{f_\nu})_\xi(\sigma_\ell(\tau)) \,G_\nu(\tau)\,\mathrm{d}\tau.
		\]
		
		Since \(f\in\mathsf{\Lambda}^1\mathscr{G}^{(s)}(M\times\mathbb{T}^m)\), for every \(\delta>0\) there exists \(C>0\), independent of \(r\), such that
		\[
		\sup_{1\leq \nu\leq n}\ \sup_{t\in K_i}\bigl|(\widehat{f_\nu})_\xi(t)\bigr|
		\leq
		Ce^{-\delta|\xi|^{\frac{1}{s}}}.
		\]
		
		Then, for each \(\delta>0\),
		\[
		\left|\int_{\tau_{r-1}}^{\tau_r}
		e^{i\xi\cdot\boldsymbol{\psi}\circ\hat\sigma_\ell}\,\sigma_\ell^*\widehat{f}_\xi\right|
		\leq
		\sum_{\nu=1}^n\int_{\tau_{r-1}}^{\tau_r}
		\bigl|(\widehat{f_\nu})_\xi(\sigma_\ell(\tau))\bigr|\,|G_\nu(\tau)|\,\mathrm{d}\tau
		\leq
		C'e^{-\delta|\xi|^{\frac{1}{s}}},
		\]
		for some \(C'>0\) independent of \(\xi\in\mathbb{Z}^m\) and \(r\in\{1,\dots,L\}\).
		Finally,
		\begin{align*}
			\left|\int_{\hat{\sigma}_\ell}e^{i\xi\cdot \boldsymbol{\psi}}\Pi^*\widehat{f}_\xi\right|
			&\leq
			\sum_{r=1}^L
			\left|\int_{\tau_{r-1}}^{\tau_r}
			e^{i\xi\cdot \boldsymbol{\psi}\circ\hat\sigma_\ell}\,\sigma_\ell^*\widehat{f}_\xi\right|
			\leq
			C''e^{-\delta|\xi|^{\frac{1}{s}}},
		\end{align*}
		where \(C''=C'L>0\).
		Then, given \(\varepsilon>0\), if we choose \(\delta = \varepsilon + \varepsilon_0\), we obtain
		\begin{equation}\label{est_uxit0_b}
			|\widehat{u}_\xi(t_0)|
			\leq
			C_0e^{-\varepsilon|\xi|^{\frac{1}{s}}},
			\qquad \forall\,\xi\in\mathbb{Z}^m,
		\end{equation}
		with \(C_0=C''C^{-1}>0\).
		
		Now we estimate the derivatives of \(\widehat{u}_\xi(t)\).
		Given \(K\Subset B\), Lemma~\ref{lemma_ADL_beurling} yields that, for every \(h>0\), there exists \(\widetilde{\varepsilon}>0\) such that
		\[
		\sup_{t\in K}\bigl|\partial_t^\alpha e^{i\xi\cdot\boldsymbol{\phi}(t)}\bigr|
		\leq
		h^{|\alpha|}\alpha!^s e^{\widetilde{\varepsilon}|\xi|^{\frac{1}{s}}},\qquad\forall\,\xi\in\mathbb{Z}^m,\ \forall\,\alpha\in\mathbb{N}_0^n.
		\]
		
		Then, given \(\varepsilon>0\), replacing \(\varepsilon\) by \(\varepsilon+\widetilde{\varepsilon}\) in \eqref{est_uxit0_b}, we obtain
		\begin{equation}\label{est_eu_b}
			\sup_{t\in K} \left|\partial_t^\alpha\!\left(e^{i\xi\cdot (\boldsymbol{\phi}(t_0)-\boldsymbol{\phi}(t))}\widehat{u}_\xi(t_0)\right)\right| \leq C_0h^{|\alpha|} \alpha!^s e^{-\varepsilon|\xi|^{\frac{1}{s}}},\qquad\forall\,\xi\in\mathbb{Z}^m,\ \forall\,\alpha\in\mathbb{N}_0^n.
		\end{equation}
		
		For the second term in \eqref{uxit0_b}, we write
		\[
		e^{-i\xi\cdot\boldsymbol{\phi}(t_0)}\int_{t_0}^t e^{i\xi\cdot\boldsymbol{\phi}}\widehat{f}_\xi
		=
		e^{-i\xi\cdot\boldsymbol{\phi}(t_0)}F_\xi,
		\]
		where \(F_\xi\) satisfies \(\mathrm{d}F_\xi = e^{i\xi\cdot\boldsymbol{\phi}}\widehat{f}_\xi\).
		
		Since \(f\in\mathsf{\Lambda}^1\mathscr{G}^{(s)}(M\times\mathbb{T}^m)\), for every \(\varepsilon,h>0\) there exists \(C_1>0\) such that
		\begin{equation}\label{est_f_b} 
			\sup_{t\in K}|\partial_t^\alpha\widehat{f}_\xi(t)|
			\leq
			C_1h^{|\alpha|}\alpha!^s e^{-\varepsilon|\xi|^{\frac{1}{s}}},\qquad\forall\,\xi\in\mathbb{Z}^m,\ \forall\,\alpha\in\mathbb{N}_0^n.
		\end{equation}
		
		Equation \eqref{est_f_b} and Lemma~\ref{lemma_ADL_beurling} yield, for every \(\varepsilon,h>0\) and for all \(\alpha\in\mathbb{N}_0^n\) and \(\xi\in\mathbb{Z}^m\) that
		\begin{align*}
			|\partial_t^\alpha\mathrm{d}F_\xi|
			&\leq
			\sum_{\beta\leq\alpha}\binom{\alpha}{\beta}
			|\partial_t^{\alpha-\beta}e^{i\xi\cdot\boldsymbol{\phi}}|\,
			|\partial_t^{\beta}\widehat{f}_\xi|\\
			&\leq
			\sum_{\beta\leq\alpha}\binom{\alpha}{\beta}
			h^{|\alpha|-|\beta|}(\alpha-\beta)!^s e^{\varepsilon_0|\xi|^{\frac{1}{s}}}
			\, C_1h^{|\beta|}\beta!^s e^{-(\varepsilon+\varepsilon_0)|\xi|^{\frac{1}{s}}}\\
			&\leq
			C_1h^{|\alpha|}\alpha!^s e^{-\varepsilon|\xi|^{\frac{1}{s}}},
		\end{align*}
		
		Hence, for every \(\varepsilon,h>0\) we can find \(C_2>0\) such that
		\[
		\sup_{t\in K}|\partial_t^\alpha F_\xi(t)|
		\leq
		C_2h^{|\alpha|}\alpha!^s e^{-\varepsilon|\xi|^{\frac{1}{s}}},\qquad\forall\,\xi\in\mathbb{Z}^m,\ \forall\,\alpha\in\mathbb{N}_0^n.
		\]
		
		Applying Lemma~\ref{lemma_ADL_beurling}, the previous estimate, and the Leibniz rule, we obtain that, for every \(\varepsilon,h>0\), there exists \(C_3>0\) such that
		\begin{equation}\label{est_eF_b}
			\sup_{t\in K}\bigl|\partial_t^\alpha(e^{-i\xi\cdot\boldsymbol{\phi}(t)}F_\xi(t))\bigr|
			\leq
			C_3h^{|\alpha|}\alpha!^s e^{-\varepsilon|\xi|^{\frac{1}{s}}},\qquad\forall\,\xi\in\mathbb{Z}^m,\ \forall\,\alpha\in\mathbb{N}_0^n.
		\end{equation}
		
		Combining \eqref{est_eu_b} and \eqref{est_eF_b}, we obtain that, for every \(\varepsilon,h>0\), there exists \(C>0\) such that
		\[
		\sup_{t\in K}|\partial_t^\alpha\widehat{u}_\xi(t)|
		\leq
		Ch^{|\alpha|}\alpha!^s e^{-\varepsilon|\xi|^{\frac{1}{s}}},\qquad\forall\,\xi\in\mathbb{Z}^m,\ \forall\,\alpha\in\mathbb{N}_0^n.
		\]
	
		Now, given any \(K\Subset M\), we can cover \(K\) by finitely many open coordinate patches \(B\) such that
		\(\Pi:\hat B\to B\) is a diffeomorphism. In particular, the previous estimates remain valid on each $B$, and hence on \(K\).
		Therefore \(u\in\mathscr{G}^{(s)}(M\times\mathbb{T}^m)\), which proves the \((s)\)-global hypoellipticity of \(\mathbb{L}\).
		\end{proof}
		
		Let us recall that \(\mathbb{L}\) is globally hypoelliptic (in the smooth sense) if
		\[
		u\in\mathscr{D}'(M\times\mathbb{T}^m),
		\quad
		\mathbb{L}u\in\mathsf{\Lambda}^1 C^\infty(M\times\mathbb{T}^m)
		\ \Longrightarrow\
		u\in C^\infty(M\times\mathbb{T}^m).
		\]
		
		Assuming that the \(1\)-forms \(\omega_1,\dots,\omega_m\) belong to \(\mathsf{\Lambda}^1\mathscr{G}^{[s]}(M)\) for a fixed \(s>1\), we obtain the following.
		
		\begin{corollary}
			Fix \(s>1\). If \(\mathbb{L}\) is globally hypoelliptic, then \(\mathbb{L}\) is \([s]\)-globally hypoelliptic.
		\end{corollary}
		
		\begin{proof}
			By \cite[Theorem 3.3]{CKMT}, if \(\mathbb{L}\) is globally hypoelliptic, then the family
			\(\boldsymbol{\omega}=(\omega_1,\dots,\omega_m)\) is neither rational nor Liouville.
			Moreover, by \cite[Proposition 3.8 (ii)]{CKMT}, the matrix of cycles \(A(\boldsymbol{\omega})\), as defined in the previous section, satisfies the following condition: there exist constants \(C,\rho>0\) such that
			\begin{equation}\label{DC}
				|\eta+A(\boldsymbol{\omega})\xi|
				\geq C(|\eta|+|\xi|)^{-\rho},
				\tag{DC}
			\end{equation}
			for all \((\xi,\eta)\in\mathbb{Z}^m\times\mathbb{Z}^d\setminus\{(0,0)\}\).
			It is clear that \eqref{DC} implies \eqref{DCs} and \eqref{DCss}, which concludes the proof.
		\end{proof}
		
	\begin{corollary}
		Fix \(s>1\) and suppose that \(\omega_1,\dots,\omega_m\in\mathsf{\Lambda}^1\mathscr{G}^{(s)}(M)\).
		If \(\mathbb{L}\) is \((s)\)-globally hypoelliptic, then \(\mathbb{L}\) is \(s\)-globally hypoelliptic.
	\end{corollary}
	
	\begin{proof}
		If \(\mathbb{L}\) is \((s)\)-globally hypoelliptic, then \(\boldsymbol{\omega}\) is not rational and \eqref{DCss} does not hold by Proposition~\ref{liou_DCs} and Theorem~\ref{gh_beurling}.
		Hence \eqref{DCs} does not hold, since \eqref{DCs} clearly implies \eqref{DCss}.
		Therefore, \(\mathbb{L}\) is \(s\)-globally hypoelliptic by Proposition~\ref{liou_DCs} and Theorem~\ref{gh_roumieu}.
	\end{proof}

	\appendix\section{Partial Fourier series in Gevrey spaces}\label{appendix}
	
	We begin by recalling some basic facts about partial Fourier series for functions in \(C^\infty(U\times\mathbb{T}^m)\), where \(U\subset\mathbb{R}^n\) is open.
	Given \(f \in C^\infty(U \times \mathbb{T}^m)\), for each \(\xi \in \mathbb{Z}^m\) we define the partial Fourier coefficient \(\widehat{f}_\xi \in C^\infty(U)\) by
	\[
	\widehat{f}_\xi(t) = \int_{\mathbb{T}^m} e^{-i \xi\cdot x}\, f(t, x)\, \mathrm{d}x, \qquad t\in U.
	\]
	
	Similarly, if \(f \in \mathscr{D}'(U \times \mathbb{T}^m)\), the partial Fourier coefficient \(\widehat{f}_\xi \in \mathscr{D}'(U)\) is defined by
	\[
	\langle \widehat{f}_\xi, \phi \rangle = \langle f, \phi \otimes e^{-i\xi\cdot x} \rangle,
	\qquad \phi \in C_c^\infty(U).
	\]
	
	\begin{theorem}
		Let \(U \subset \mathbb{R}^n\) be open and \(u \in \mathscr{D}'(U \times \mathbb{T}^m)\).
		Then \(u \in C^\infty(U \times \mathbb{T}^m)\) if and only if the following conditions hold:
		\begin{enumerate}
			\item \(\widehat{u}_\xi \in C^\infty(U)\) for all \(\xi \in \mathbb{Z}^m\);
			
			\item for every \(K \Subset U\), \(\alpha \in \mathbb{N}_0^n\), and \(N \in \mathbb{N}_0\), there exists \(C > 0\) such that
			\[
			\sup_{t \in K} |\partial_t^\alpha \widehat{u}_\xi(t)|
			\leq C (1 + |\xi|)^{-N},
			\qquad \forall\, \xi \in \mathbb{Z}^m.
			\]
		\end{enumerate}
		
		Under these conditions, we have
		\[
		u(t, x) = \sum_{\xi \in \mathbb{Z}^m} \widehat{u}_\xi(t)\, e^{i\xi\cdot x},
		\]
		with uniform convergence on compact subsets of \(U \times \mathbb{T}^m\).
	\end{theorem}
	
	One can show that these constructions are invariant under changes of variables (see \cite{ADL2023gh}).
	Hence, they extend to globally defined partial Fourier coefficients and partial Fourier series for smooth functions on \(M\times\mathbb{T}^m\), where \(M\) is a smooth paracompact manifold.
	
	Similarly, consider the space \(\mathsf{\Lambda}^{0,1} C^\infty(M \times \mathbb{T}^m)\) of smooth \(1\)-forms on \(M \times \mathbb{T}^m\) with no \(\mathrm{d}x\)-components, that is, the space of smooth \(1\)-forms \(f\) such that, in any coordinate system \((U,t)\) on \(M\), one can write
	\[
	f|_{U \times \mathbb{T}^m} = \sum_{\nu=1}^n f_\nu(t, x) \, \mathrm{d}t_\nu.
	\]
	
	In each coordinate chart \((U,t)\) on \(M\), we define the Fourier coefficient \(\widehat{f}_\xi\) of
	\(f \in \mathsf{\Lambda}^{0,1} C^\infty(M \times \mathbb{T}^m)\) by
	\[
	\widehat{f}_\xi = \sum_{\nu=1}^n (\widehat{f_\nu})_\xi \, \mathrm{d}t_\nu \in \mathsf{\Lambda}^1 C^\infty(U).
	\]
	
	Since the local partial Fourier coefficients are invariant under diffeomorphisms, the coefficients
	\(\widehat{f}_\xi\in\mathsf{\Lambda}^{1}C^\infty(M)\) are globally defined, and we have
	\[
	f(t,x) = \sum_{\xi \in \mathbb{Z}^m} \widehat{f}_\xi(t)\, e^{i \xi \cdot x},
	\]
	with uniform convergence on compact subsets. Moreover, the regularity of \(f\) is characterized by the decay properties of the Fourier coefficients \(\widehat{f}_\xi\).
	
	\begin{proposition}\label{lemma_uxieq0}
		Let \(u \in \mathscr{D}'(M \times \mathbb{T}^m)\). Then \(u = 0\) if and only if \(\widehat{u}_\xi = 0\) for all \(\xi \in \mathbb{Z}^m\).
	\end{proposition}
	\begin{proof}
		See \cite[Lemma 5.1]{ADL2023gh}.
	\end{proof}
	
	We now introduce partial Fourier series for Gevrey functions (and forms) on \(U\times\mathbb{T}^m\), where \(U\subset\mathbb{R}^n\) is open.
	First, note that if \(u\in\mathscr{D}'(M\times\mathbb{T}^m)\) satisfies
	\(\mathbb{L}u\in\mathsf{\Lambda}^1C^\infty(M\times\mathbb{T}^m)\), then the partial Fourier coefficients \(\widehat{u}_\xi\) belong to \(C^\infty(M)\) for every \(\xi\in\mathbb{Z}^m\), see \cite[Corollay 2.4]{CKMT}.
	Next, we prove two auxiliary results, which are variations of Lemmas~1.6.2 and~1.6.3 in \cite{Rod_Gevrey}.
	
	\begin{lemma}\label{lemma_fourier_r}
		Let \(f\in C^\infty(U\times\mathbb{T}^m)\), and fix \(K\Subset U\) and \(s\geq 1\).
		The following statements are equivalent:
		\begin{enumerate}
			\item There exist \(C,h,\varepsilon>0\) such that
			\[
			\sup_{t\in K}|\partial_t^\alpha\widehat{f}_\xi(t)|
			\leq Ch^{|\alpha|}\alpha!^s e^{-\varepsilon|\xi|^{\frac{1}{s}}}, \qquad\forall\, \xi\in\mathbb{Z}^m,\ \forall\, \alpha\in\mathbb{N}_0^n.
			\]
			
			\item There exist \(C,h>0\) such that
			\[
			\sup_{t\in K}|\partial_t^\alpha\widehat{f}_\xi(t)|
			\leq h^{|\alpha|}\alpha!^s\, C^{N+1}N!\,|\xi|^{-\frac{N}{s}}, \qquad\forall\, \xi\in\mathbb{Z}^m,\ \forall\, \alpha\in\mathbb{N}_0^n,\ \forall\, N\in\mathbb{N}.
			\]

			\item There exist \(C,h>0\) such that
			\[
			\sup_{t\in K}|\partial_t^\alpha\widehat{f}_\xi(t)|
			\leq h^{|\alpha|}\alpha!^s\, C(CN)^N\,|\xi|^{-\frac{N}{s}}, \qquad\forall\, \xi\in\mathbb{Z}^m,\ \forall\, \alpha\in\mathbb{N}_0^n,\ \forall\, N\in\mathbb{N}.
			\]
			
		\end{enumerate}
	\end{lemma}

\begin{proof}
	\noindent
	\((1)\Rightarrow(2)\).
	There exist \(C_0,h,\varepsilon>0\) such that
	\begin{align*}
		C_0h^{|\alpha|}\alpha!^s
		&\geq
		\sup_{t\in K}|\partial_t^\alpha\widehat{f}_\xi(t)|\,e^{\varepsilon|\xi|^{\frac{1}{s}}}
		=
		\sup_{t\in K}|\partial_t^\alpha\widehat{f}_\xi(t)|
		\sum_{N=0}^{\infty}\frac{\varepsilon^N|\xi|^{\frac{N}{s}}}{N!}\\
		&\geq
		\sup_{t\in K}|\partial_t^\alpha\widehat{f}_\xi(t)|
		\frac{\varepsilon^{N}}{N!}\,|\xi|^{\frac{N}{s}},
	\end{align*}
	for every \(\alpha\in\mathbb{N}_0^n\), \(\xi\in\mathbb{Z}^m\setminus\{0\}\), and \(N\in\mathbb{N}\).
	Therefore,
	\[
	\sup_{t\in K}|\partial_t^\alpha\widehat{f}_\xi(t)|
	\leq
	C_0h^{|\alpha|}\alpha!^s\,\varepsilon^{-N}N!\,|\xi|^{-\frac{N}{s}}.
	\]
	
	Taking \(C=\max\{1,C_0,\varepsilon^{-1}\}\), we obtain \((2)\).
	
	\medskip
	\noindent
	\((2)\Rightarrow(1)\).
	There exist \(C_0,h>0\) such that
	\[
	\sup_{t\in K}|\partial_t^\alpha\widehat{f}_\xi(t)|
	\leq
	h^{|\alpha|}\alpha!^s\, C_0^{N+1}N!\,|\xi|^{-\frac{N}{s}},\qquad\forall\,\xi\in\mathbb{Z}^m\setminus\{0\},\ \forall\,\alpha\in\mathbb{N}_0^n,\ \forall\, N\in\mathbb{N}.
	\]
	
	Equivalently,
	\begin{equation*}
		\sup_{t\in K}|\partial_t^\alpha\widehat{f}_\xi(t)|
		\frac{C_0^{-N}}{N!}\,|\xi|^{\frac{N}{s}}
		\leq
		C_0h^{|\alpha|}\alpha!^s.
	\end{equation*}
	
	Let \(\varepsilon=(2C_0)^{-1}\). Then
	\begin{equation*}
		\sup_{t\in K}|\partial_t^\alpha\widehat{f}_\xi(t)|
		\frac{\varepsilon^{N}}{N!}\,|\xi|^{\frac{N}{s}}
		\leq
		C_0h^{|\alpha|}\alpha!^s\,2^{-N}, \qquad\forall\,\xi\in\mathbb{Z}^m\setminus\{0\},\ \forall\,\alpha\in\mathbb{N}_0^n,\ \forall\, N\in\mathbb{N}.
	\end{equation*}
	
	Summing both sides over \(N\) yields
	\begin{equation*}
		\sup_{t\in K}|\partial_t^\alpha\widehat{f}_\xi(t)|
		\sum_{N=0}^\infty \frac{\varepsilon^{N}}{N!}\,|\xi|^{\frac{N}{s}}
		\leq
		C_0h^{|\alpha|}\alpha!^s\sum_{N=0}^\infty 2^{-N}
		=2C_0h^{|\alpha|}\alpha!^s,
	\end{equation*}
	that is,
	\begin{equation*}
		\sup_{t\in K}|\partial_t^\alpha\widehat{f}_\xi(t)|
		e^{\varepsilon|\xi|^{\frac{1}{s}}}
		\leq
		Ch^{|\alpha|}\alpha!^s,\qquad\forall\,\xi\in\mathbb{Z}^m\setminus\{0\},\ \forall\,\alpha\in\mathbb{N}_0^n,
	\end{equation*}
	where \(C=2C_0\). This proves \((1)\).
	
	\medskip
	\noindent
	\((2)\Leftrightarrow(3)\).
	This is a direct consequence of the inequalities
	\(N!\leq N^N \leq e^N N!\), valid for all \(N\in\mathbb{N}\).
\end{proof}

\begin{lemma}\label{lemma_fourier_aux_r}
	Let \(N\in\mathbb{N}\) and \(s\geq 1\).
	Suppose that \(f\in C^\infty(U\times\mathbb{T}^m)\) and that, given \(K\Subset U\), there exist \(C,h>0\) such that
	\[
	\sup_{t\in K}\sup_{x\in\mathbb{T}^m}|\partial_t^\alpha\partial_x^\beta f(t,x)|
	\leq
	Ch^{|\alpha|+|\beta|}\alpha!^s\, M^{s|\beta|},\qquad \forall\,\alpha\in\mathbb{N}_0^n,\ \forall\,\beta\in\mathbb{N}_0^m,\ |\beta|\leq M,
	\]
	where \(M\) is the smallest integer such that \(M\geq N/s\). Then
	\[
	\sup_{t\in K}|\partial_t^\alpha\widehat{f}_\xi(t)|
	\leq
	h^{|\alpha|}\alpha!^s\, C'(C'N)^N\,|\xi|^{-\frac{N}{s}},\qquad \forall\,\xi\in\mathbb{Z}^m\setminus\{0\},\  \forall\,\alpha\in\mathbb{N}_0^n,
	\]
	where \(C'>0\) is a constant independent of \(N\).
\end{lemma}

\begin{proof}
	There exist \(C,h>0\) such that, for every \(\alpha\in\mathbb{N}_0^n\) and \(\beta\in\mathbb{N}_0^m\) with \(|\beta|\leq M\),
	\[
	\sup_{t\in K}|\xi^\beta\partial_t^\alpha\widehat{f}_\xi(t)|
	\leq
	\sup_{t\in K}\int_{\mathbb{T}^m}|\partial_t^\alpha\partial_x^\beta f(t,x)|\,\mathrm{d}x
	\leq
	Ch^{|\alpha|+|\beta|}\alpha!^s\, M^{s|\beta|}.
	\]
	
	Let \(N\in\mathbb{N}\) and let \(M\) be the smallest integer such that \(M\geq N/s\).
	Then, for every \(\alpha\in\mathbb{N}_0^n\) and \(\beta\in\mathbb{N}_0^m\) with \(|\beta|\leq M\),
	\begin{align*}
		|\xi|^{\frac{N}{s}}|\partial_t^\alpha\widehat{f}_\xi(t)|
		&\leq
		|\xi|^M|\partial_t^\alpha \widehat{f}_\xi(t)|\ \leq\
		\sum_{|\gamma|=M}\frac{M!}{\gamma!}\,|\xi^\gamma|\,|\partial_t^\alpha\widehat{f}_\xi(t)|\\
		&\leq
		\sum_{|\gamma|=M}\frac{M!}{\gamma!}\, Ch^{|\alpha|+|\gamma|}\alpha!^s\, M^{s|\gamma|}\ \leq\ 
		Ch^{|\alpha|}\alpha!^s\,(mh)^M\, M^{sM}.
	\end{align*}
	
	Since \(N/s\leq M\leq N/s+1\), we have \(sM\leq N+s\) and \(M\leq N/s+1\leq N+1\leq 2N\). Hence,
	\[
	M^{sM}\leq M^{s}M^N \leq s^s e^M M^N \leq s^s(2e^{2})^{N}N^N.
	\]
	
	Therefore,
	\begin{align*}
		Ch^{|\alpha|}\alpha!^s\,(mh)^M\, M^{sM}
		&\leq
		Ch^{|\alpha|}\alpha!^s\,(mh)^M\, s^s(2e^{2})^{N}N^N\\
		&\leq
		h^{|\alpha|}\alpha!^s\, C'(C'N)^N,
	\end{align*}
	for suitable constants \(C',C>0\), with \(C'\) independent of \(N\).
	This concludes the proof.
\end{proof}

We will also need the following estimate.

\begin{lemma}\label{xi_beta_exp}
	Given \(\xi\in\mathbb{Z}^m\), \(\beta\in\mathbb{N}_0^m\), \(s\geq 1\), and \(\delta>0\), we have
	\[
	|\xi|^{|\beta|}
	\leq
	\left(s^s m^s \delta^{-s}\right)^{|\beta|}\beta!^s.
	\]
\end{lemma}
\begin{proof}
	It follows from \cite[Lemma A.5]{KowTok_arXiv} and the inequality \( |\beta|!\leq m^{|\beta|}\beta! \).
\end{proof}

We can now characterize Gevrey functions of Roumieu type by the decay properties of their partial Fourier coefficients:

\begin{theorem}
	Let \(u\in\mathscr{D}'(U\times\mathbb{T}^m)\) be such that \(\widehat{u}_\xi\in C^\infty(U)\) for every \(\xi\in\mathbb{Z}^m\), and fix \(s\geq 1\).
	Then the following statements are equivalent:
	\begin{enumerate}
		\item \(u\in \mathscr{G}^s(U\times\mathbb{T}^m)\);
		\item for each \(K\Subset U\), there exist constants \(C,h,\varepsilon>0\) such that
		\[
		\sup_{t\in K}|\partial_t^\alpha\widehat{u}_\xi(t)|
		\leq
		Ch^{|\alpha|}\alpha!^s e^{-\varepsilon|\xi|^{\frac{1}{s}}},
		\qquad
		\forall\,\xi\in\mathbb{Z}^m,\ \forall\, \alpha\in\mathbb{N}_0^n.
		\]
	\end{enumerate}
	
	Moreover, these conditions are invariant under analytic changes of variables.
\end{theorem}

\begin{proof}
	\noindent
	\((1)\Rightarrow(2)\).
	Assume that \(u\in\mathscr{G}^{s}(U\times\mathbb{T}^m)\).
	Then, given \(K\Subset U\), there exist \(C,h>0\) such that, for every \(\alpha\in\mathbb{N}_0^n\), \(\beta\in\mathbb{N}_0^m\), and \(\xi\in\mathbb{Z}^m\),
	\begin{align*}
		\sup_{t\in K}|\xi^\beta\partial_t^\alpha\widehat{u}_\xi(t)|
		&\leq
		\int_{\mathbb{T}^m}|\partial_t^\alpha\partial_x^\beta u(t,x)|\,\mathrm{d}x
		\leq
		Ch^{|\alpha|+|\beta|}\alpha!^s\beta!^s\\
		&\leq
		Ch^{|\alpha|+|\beta|}\alpha!^s|\beta|^{s|\beta|}.
	\end{align*}
	
	Fix \(N\in\mathbb{N}\), and let \(M\) be the smallest integer such that \(N/s\leq M\).
	Then, by Lemma~\ref{lemma_fourier_aux_r}, we obtain
	\[
	\sup_{t\in K}|\partial_t^\alpha\widehat{u}_\xi(t)|
	\leq
	h^{|\alpha|}\alpha!^s\, C'(C'N)^N|\xi|^{-\frac{N}{s}},\qquad\forall\,\xi\in\mathbb{Z}^m\setminus\{0\},\ \forall\,\alpha\in\mathbb{N}_0^n,\ \forall\, N\in\mathbb{N},
	\]
	where \(C'>0\) is independent of \(N\).
	By Lemma~\ref{lemma_fourier_r}, condition \((2)\) follows.
	
	\medskip
	\noindent
	\((2)\Rightarrow(1)\).
	Assume that \((2)\) holds.
	Then the series
	\[
	\sum_{\xi\in\mathbb{Z}^m}\widehat{u}_\xi(t)e^{i\xi\cdot x}
	\]
	converges uniformly on compact subsets to a smooth function \(g\in C^\infty(U\times\mathbb{T}^m)\), and
	\[
	\partial_t^\alpha\partial_x^\beta g(t,x)
	=
	\sum_{\xi\in\mathbb{Z}^m}\partial_t^\alpha\widehat{u}_\xi(t)(i\xi)^\beta e^{i\xi\cdot x},
	\]
	for every \(\alpha\in\mathbb{N}_0^n\) and \(\beta\in\mathbb{N}_0^m\), with uniform convergence on compact subsets.
	In particular, given \(K\Subset U\), there exist \(C,h_0,\varepsilon>0\) such that
	\begin{align*}
		\sup_{t\in K}|\partial_t^\alpha\partial_x^\beta g(t,x)|
		&\leq
		\sup_{t\in K}\sum_{\xi\in\mathbb{Z}^m}|\xi|^{|\beta|}|\partial_t^\alpha\widehat{u}_\xi(t)|\\
		&\leq
		Ch_0^{|\alpha|}\alpha!^s\sum_{\xi\in\mathbb{Z}^m}|\xi|^{|\beta|}e^{-\varepsilon|\xi|^{\frac{1}{s}}}\\
		&\leq
		Ch_0^{|\alpha|}\alpha!^s\sum_{\xi\in\mathbb{Z}^m}\left(\frac{2^ss^sm^s}{\varepsilon^s}\right)^{|\beta|}\beta!^s
		e^{\frac{\varepsilon}{2}|\xi|^{\frac{1}{s}}}e^{-\varepsilon|\xi|^{\frac{1}{s}}},
	\end{align*}
	where we used Lemma~\ref{xi_beta_exp} with \(\delta=\varepsilon/2\).
	Therefore,
	\[
	\sup_{t\in K}|\partial_t^\alpha\partial_x^\beta g(t,x)|
	\leq
	C'h^{|\alpha|+|\beta|}\alpha!^s\beta!^s,
	\]
	where
	\(C'=C\sum_{\xi\in\mathbb{Z}^m}e^{-\frac{\varepsilon}{2}|\xi|^{\frac{1}{s}}}\)
	and
	\(h=\max\{1,h_0,2^ss^sm^s\varepsilon^{-s}\}\).
	Hence \(g\in\mathscr{G}^{s}(U\times\mathbb{T}^m)\).
	
	It remains to show that \(g=u\).
	By uniform convergence, we may integrate term by term. Thus, for every \(\eta\in\mathbb{Z}^m\),
	\begin{align*}
		\widehat{g}_\eta(t)
		&=
		\int_{\mathbb{T}^m} g(t,x)e^{-i\eta\cdot x}\,\mathrm{d}x
		=
		\sum_{\xi\in\mathbb{Z}^m}\int_{\mathbb{T}^m}\widehat{u}_\xi(t)e^{-i(\eta-\xi)\cdot x}\,\mathrm{d}x\\
		&=
		\sum_{\xi\in\mathbb{Z}^m}\widehat{u}_\xi(t)\int_{\mathbb{T}^m} e^{-i(\eta-\xi)\cdot x}\,\mathrm{d}x
		=
		\widehat{u}_\eta(t).
	\end{align*}
	
	By Lemma \ref{lemma_uxieq0} we conclude that \(g=u\), which completes the proof.
\end{proof}

We can obtain a similar characterization for Gevrey functions of Beurling type in terms of their partial Fourier coefficients.
Before that, we prove the Beurling versions of Lemmas~\ref{lemma_fourier_r} and~\ref{lemma_fourier_aux_r}.

\begin{lemma}\label{lemma_fourier_b}
	Let \(f\in C^\infty(U\times\mathbb{T}^m)\), and fix \(K\Subset U\) and \(s\geq 1\).
	The following statements are equivalent:
	\begin{enumerate}
		\item For every \(\varepsilon,h>0\), there exists \(C>0\) such that
		\[
		\sup_{t\in K}|\partial_t^\alpha\widehat{f}_\xi(t)|
		\leq
		Ch^{|\alpha|}\alpha!^s e^{-\varepsilon|\xi|^{\frac{1}{s}}},\qquad\forall\, \xi\in\mathbb{Z}^m,\ \forall\, \alpha\in\mathbb{N}_0^n.
		\]
		
		\item For every \(h_1,h_2>0\), there exists \(C>0\) such that
		\[
		\sup_{t\in K}|\partial_t^\alpha\widehat{f}_\xi(t)|
		\leq
		Ch_1^{|\alpha|}\alpha!^s\, h_2^N N!\,|\xi|^{-\frac{N}{s}}, \qquad \forall\, \xi\in\mathbb{Z}^m\setminus\{0\},\ \forall\,\alpha\in\mathbb{N}_0^n,\ \forall\,N\in\mathbb{N}. 
		\]
		
		\item For every \(h_1,h_2>0\), there exists \(C>0\) such that
		\[
		\sup_{t\in K}|\partial_t^\alpha\widehat{f}_\xi(t)|
		\leq
		Ch_1^{|\alpha|}\alpha!^s\, (h_2N)^N\,|\xi|^{-\frac{N}{s}},\qquad \forall\, \xi\in\mathbb{Z}^m\setminus\{0\},\ \forall\,\alpha\in\mathbb{N}_0^n,\ \forall\, N\in\mathbb{N}.
		\]
	\end{enumerate}
\end{lemma}

\begin{proof}
	\noindent
	\((1)\Rightarrow(2)\).
	Given \(\varepsilon,h>0\), there exists \(C>0\) such that
	\begin{align*}
		Ch^{|\alpha|}\alpha!^s
		&\geq
		\sup_{t\in K}|\partial_t^\alpha\widehat{f}_\xi(t)|\,e^{\varepsilon|\xi|^{\frac{1}{s}}}\ =\ 
		\sup_{t\in K}|\partial_t^\alpha\widehat{f}_\xi(t)|
		\sum_{N=0}^{\infty}\frac{\varepsilon^N|\xi|^{\frac{N}{s}}}{N!}\\
		&\geq
		\sup_{t\in K}|\partial_t^\alpha\widehat{f}_\xi(t)|
		\frac{\varepsilon^{N}}{N!}\,|\xi|^{\frac{N}{s}},
	\end{align*}
	for every \(\alpha\in\mathbb{N}_0^n\), \(\xi\in\mathbb{Z}^m\setminus\{0\}\), and \(N\in\mathbb{N}\).
	Thus,
	\[
	\sup_{t\in K}|\partial_t^\alpha\widehat{f}_\xi(t)|
	\leq
	Ch^{|\alpha|}\alpha!^s\, \varepsilon^{-N}N!\,|\xi|^{-\frac{N}{s}}.
	\]
	
	Given \(h_1,h_2>0\), choosing \(h=h_1\) and \(\varepsilon=h_2^{-1}\) yields \((2)\).
	
	\medskip
	\noindent
	\((2)\Rightarrow(1)\).
	Given \(h_1,h_2>0\), there exists \(C>0\) such that
	\[
	\sup_{t\in K}|\partial_t^\alpha\widehat{f}_\xi(t)|
	\frac{h_2^{-N}}{N!}\,|\xi|^{\frac{N}{s}}
	\leq
	Ch_1^{|\alpha|}\alpha!^s,\qquad\forall\,\xi\in\mathbb{Z}^m\setminus\{0\},\ \forall\,\alpha\in\mathbb{N}_0^n,\ \forall\, N\in\mathbb{N}.
	\]
	
	Given \(\varepsilon,h>0\), take \(h_1=h\) and \(h_2=(2\varepsilon)^{-1}\). Then
	\[
	\sup_{t\in K}|\partial_t^\alpha\widehat{f}_\xi(t)|
	\frac{\varepsilon^{N}}{N!}\,|\xi|^{\frac{N}{s}}
	\leq
	Ch^{|\alpha|}\alpha!^s\,2^{-N},\qquad\forall\,\xi\in\mathbb{Z}^m\setminus\{0\},\ \forall\,\alpha\in\mathbb{N}_0^n,\ \forall\, N\in\mathbb{N}.
	\]
	
	Summing over \(N\) yields
	\[
	\sup_{t\in K}|\partial_t^\alpha\widehat{f}_\xi(t)|
	e^{\varepsilon|\xi|^{\frac{1}{s}}}
	\leq
	2Ch^{|\alpha|}\alpha!^s,\qquad\forall\,\xi\in\mathbb{Z}^m\setminus\{0\},\ \forall\,\alpha\in\mathbb{N}_0^n.
	\]
	
	This proves \((1)\).
	
	\medskip
	\noindent
	\((2)\Leftrightarrow(3)\).
	This is a direct consequence of the inequalities
	\(N!\leq N^N \leq e^N N!\), valid for all \(N\in\mathbb{N}\).
\end{proof}

\begin{lemma}\label{lemma_fourier_aux_b}
	Let \(N\in\mathbb{N}\) and \(s>1\).
	Suppose that \(f\in C^\infty(U\times\mathbb{T}^m)\) and that, given \(K\Subset U\), for every \(h>0\) there exists \(C>0\) such that
	\[
	\sup_{t\in K}\sup_{x\in\mathbb{T}^m}|\partial_t^\alpha\partial_x^\beta f(t,x)|
	\leq
	Ch^{|\alpha|+|\beta|}\alpha!^s\, M^{s|\beta|},\qquad \forall\,\alpha\in\mathbb{N}_0^n,\ \forall\,\beta\in\mathbb{N}_0^m,\ |\beta|\leq M,
	\]
	where \(M\) is the smallest integer such that \(M\geq N/s\). Then
	\[
	\sup_{t\in K}|\partial_t^\alpha\widehat{f}_\xi(t)|
	\leq
	C'h^{|\alpha|}\alpha!^s (hN)^N\, |\xi|^{-\frac{N}{s}},\qquad \forall\,\xi\in\mathbb{Z}^m\setminus\{0\},\  \forall\,\alpha\in\mathbb{N}_0^n,
	\]
	where \(C'>0\) is a constant independent of \(N\).
\end{lemma}

\begin{proof}
	Given \(\alpha\in\mathbb{N}_0^n\) and \(\beta\in\mathbb{N}_0^m\), for every \(h_0>0\) there exists \(C>0\) such that
	\[
	\sup_{t\in K}|\xi^\beta\partial_t^\alpha\widehat{f}_\xi(t)|
	\leq
	\sup_{t\in K}\int_{\mathbb{T}^m}|\partial_t^\alpha\partial_x^\beta f(t,x)|\,\mathrm{d}x
	\leq
	Ch_0^{|\alpha|+|\beta|}\alpha!^s\, M^{s|\beta|},
	\]
	where \(M\) satisfies \(|\beta|\leq M\).
	
	Let \(N\in\mathbb{N}\) and let \(M\) be the smallest integer such that \(M\geq N/s\).
	Then, for every \(\beta\in\mathbb{N}_0^m\) with \(|\beta|\leq M\),
	\begin{align*}
		|\xi|^{\frac{N}{s}}|\partial_t^\alpha\widehat{f}_\xi(t)|
		&\leq
		|\xi|^M|\partial_t^\alpha \widehat{f}_\xi(t)|\ \leq\ 
		\sum_{|\gamma|=M}\frac{M!}{\gamma!}\,|\xi^\gamma|\,|\partial_t^\alpha\widehat{f}_\xi(t)|\\
		&\leq
		\sum_{|\gamma|=M}\frac{M!}{\gamma!}\, Ch_0^{|\alpha|+|\gamma|}\alpha!^s\, M^{s|\gamma|}\ \leq\
		Ch_0^{|\alpha|}\alpha!^s\,(mh_0)^M\, M^{sM},
	\end{align*}
	for every \(h_0>0\) and \(\alpha\in\mathbb{N}_0^n\), where \(C>0\) depends only on \(h_0>0\).
	
	Since \(N/s\leq M\leq N/s+1\), we have \(sM\leq N+s\) and \(M\leq N/s+1\leq N+1\leq 2N\). Hence,
	\[
	M^{sM}\leq M^{s}M^N \leq s^s e^M M^N \leq s^s(2e^{2})^{N}N^N.
	\]
	
	Given \(h>0\), we distinguish two cases.
	
	\smallskip
	\noindent\emph{Case \(h\geq 1\):}
	Set \(h_0=h(m2^se^{2s})^{-1}\). Then
	\begin{align*}
		Ch_0^{|\alpha|}\alpha!^s(mh_0)^M M^{sM}
		&\leq
		Ch^{|\alpha|}\alpha!^s\left(\frac{mh}{m2^se^{2s}}\right)^M s^s(2e^{2})^{N}N^N\\
		&\leq
		(Cs^s)h^{|\alpha|}\alpha!^s\, h^{M} N^N\\
		&\leq
		(Cs^s)h^{|\alpha|}\alpha!^s\, h^{N+1} N^N\\
		&\leq
		C'h^{|\alpha|}\alpha!^s (hN)^{N},
	\end{align*}
	where \(C'=Cs^sh\), and we used
	\begin{equation}\label{m2se2s} 
		\frac{1}{(m2^se^{2s})^{|\alpha|}}\leq 1,\qquad\forall\,\alpha\in\mathbb{N}_0^n.
	\end{equation}
	
	\smallskip
	\noindent\emph{Case \(h<1\):}
	Set \(h_0=h^s(m2^se^{2s})^{-1}\). Then
	\begin{align*}
		Ch_0^{|\alpha|}\alpha!^s(mh_0)^M M^{sM}
		&\leq
		Ch^{|\alpha|}\alpha!^s\left(\frac{mh^s}{m2^se^{2s}}\right)^M s^s(2e^{2})^{N}N^N\\
		&\leq
		(Cs^s)h^{s|\alpha|}\alpha!^s\, h^{sM} N^N\\
		&\leq
		(Cs^s)h^{|\alpha|}\alpha!^s\, h^{N+s} N^N\\
		&\leq
		C'h^{|\alpha|}\alpha!^s (hN)^{N},
	\end{align*}
	where \(C'=Cs^sh^s\), and we used \eqref{m2se2s} together with the fact that
	\(h^{s|\alpha|}\leq h^{|\alpha|}\) for every \(\alpha\in\mathbb{N}_0^n\), since \(h<1\) and \(s\geq 1\).
	This completes the proof.
\end{proof}

\begin{theorem}
	Let \(u\in\mathscr{D}'(U\times\mathbb{T}^m)\) be such that \(\widehat{u}_\xi\in C^\infty(U)\) for every \(\xi\in\mathbb{Z}^m\), and fix \(s\geq 1\).
	Then the following statements are equivalent:
	\begin{enumerate}
		\item\label{p:gequiv-a} 
		\(u\in \mathscr{G}^{(s)}(U\times\mathbb{T}^m)\);
		\item\label{p:gequiv-b} for each \(K\Subset U\) and each \(\varepsilon,h>0\), there exists \(C>0\) such that
		\[
		\sup_{t\in K}|\partial_t^\alpha\widehat{u}_\xi(t)|
		\leq
		Ch^{|\alpha|}\alpha!^s e^{-\varepsilon|\xi|^{\frac{1}{s}}},
		\qquad \forall\,\xi\in\mathbb{Z}^m,\ \forall\, \alpha\in\mathbb{N}_0^n.
		\]
	\end{enumerate}
	 
	Moreover, the conditions in (\ref{p:gequiv-a}) and (\ref{p:gequiv-b}) 
	are invariant under analytic changes of variables.
\end{theorem}

\begin{proof} The invariance of the conditions in the statement under
analytic changes of variables is immediate. Let us show their equivalence.

	\noindent
	\((1)\Rightarrow(2)\).
	Assume that \(u\in\mathscr{G}^{(s)}(U\times\mathbb{T}^m)\).
	Then, given \(K\Subset U\) and \(h>0\), there exists \(C>0\) such that, for every \(\alpha\in\mathbb{N}_0^n\), \(\beta\in\mathbb{N}_0^m\), and \(\xi\in\mathbb{Z}^m\),
	\begin{align*}
		\sup_{t\in K}|\xi^\beta\partial_t^\alpha\widehat{u}_\xi(t)|
		&\leq
		\int_{\mathbb{T}^m}|\partial_t^\alpha\partial_x^\beta u(t,x)|\,\mathrm{d}x
		\leq
		Ch^{|\alpha|+|\beta|}\alpha!^s\beta!^s\\
		&\leq
		Ch^{|\alpha|+|\beta|}\alpha!^s|\beta|^{s|\beta|}.
	\end{align*}
	
	Fix \(N\in\mathbb{N}\), and let \(M\) be the smallest integer such that \(N/s\leq M\).
	By Lemma~\ref{lemma_fourier_aux_b}, we obtain
	\[
	\sup_{t\in K}|\partial_t^\alpha\widehat{u}_\xi(t)|
	\leq
	Ch^{|\alpha|}\alpha!^s(hN)^N|\xi|^{-\frac{N}{s}},
	\]
	for every \(\alpha\in\mathbb{N}_0^n\), \(\xi\in\mathbb{Z}^m\setminus\{0\}\), and \(N\in\mathbb{N}\).
	By Lemma~\ref{lemma_fourier_b}, condition \((2)\) follows.
	
	\medskip
	\noindent
	\((2)\Rightarrow(1)\).
	Assume that \((2)\) holds.
	Then the series
	\[
	\sum_{\xi\in\mathbb{Z}^m}\widehat{u}_\xi(t)e^{i\xi\cdot x}
	\]
	converges uniformly on compact subsets to a smooth function \(g\in C^\infty(U\times\mathbb{T}^m)\), and
	\[
	\partial_t^\alpha\partial_x^\beta g(t,x)
	=
	\sum_{\xi\in\mathbb{Z}^m}\partial_t^\alpha\widehat{u}_\xi(t)(i\xi)^\beta e^{i\xi\cdot x},
	\]
	for every \(\alpha\in\mathbb{N}_0^n\) and \(\beta\in\mathbb{N}_0^m\), with uniform convergence on compact subsets.
	In particular, given \(K\Subset U\), for every \(\varepsilon,h>0\),
	\begin{align*}
		\sup_{t\in K}|\partial_t^\alpha\partial_x^\beta g(t,x)|
		&\leq
		\sup_{t\in K}\sum_{\xi\in\mathbb{Z}^m}|\xi|^{|\beta|}\,|\partial_t^\alpha\widehat{u}_\xi(t)|\\
		&\leq
		Ch^{|\alpha|}\alpha!^s\sum_{\xi\in\mathbb{Z}^m}|\xi|^{|\beta|}e^{-\varepsilon|\xi|^{\frac{1}{s}}}\\
		&\leq
		Ch^{|\alpha|}\alpha!^s\sum_{\xi\in\mathbb{Z}^m}\left(\frac{2^ss^sm^s}{\varepsilon^s}\right)^{|\beta|}\beta!^s
		e^{\frac{\varepsilon}{2}|\xi|^{\frac{1}{s}}}e^{-\varepsilon|\xi|^{\frac{1}{s}}},
	\end{align*}
	where we used Lemma~\ref{xi_beta_exp} with \(\delta=\varepsilon/2\).
	Choosing \(\varepsilon=\dfrac{2^ss^sm^s}{h^{1/s}}\), we obtain
	\[
	\sup_{t\in K}|\partial_t^\alpha\partial_x^\beta g(t,x)|
	\leq
	C'h^{|\alpha|+|\beta|}\alpha!^s\beta!^s,
	\]
	where \(C'=C\sum_{\xi\in\mathbb{Z}^m}e^{-\frac{\varepsilon}{2}|\xi|^{\frac{1}{s}}}\).
	Hence \(g\in\mathscr{G}^{(s)}(U\times\mathbb{T}^m)\).
	
	It remains to show that \(g=u\).
	By uniform convergence, we may integrate term by term. Thus, for every \(\eta\in\mathbb{Z}^m\),
	\[
	\widehat{g}_\eta(t)
	=
	\sum_{\xi\in\mathbb{Z}^m}\int_{\mathbb{T}^m}\widehat{u}_\xi(t)e^{-i(\eta-\xi)\cdot x}\,\mathrm{d}x
	=
	\sum_{\xi\in\mathbb{Z}^m}\widehat{u}_\xi(t)\int_{\mathbb{T}^m} e^{-i(\eta-\xi)\cdot x}\,\mathrm{d}x
	=
	\widehat{u}_\eta(t),
	\]
	which concludes the proof.
\end{proof}

By invariance under analytic diffeomorphisms, these constructions extend to a smooth paracompact manifold \(M\) endowed with a countable, locally finite, analytic atlas, as above.

If \(f\) is a Gevrey \(1\)-form of Roumieu or Beurling type, then in every analytic chart we can identify its partial Fourier coefficient \(\widehat{f}_\xi(t)\), for each \(\xi\in\mathbb{Z}^m\), with the vector
\(\bigl((\widehat{f_1})_\xi(t),\dots,(\widehat{f_n})_\xi(t)\bigr)\).
Accordingly, the estimate
\[
|\partial_t^{\alpha}\widehat{f}_\xi(t)|
\leq
Ch^{|\alpha|}\alpha!^s e^{-\varepsilon|\xi|^{\frac{1}{s}}}
\]
means that the same bound holds for each component \(|\partial_t^{\alpha}(\widehat{f_\nu})_\xi(t)|\).

\section*{Acknowledgments}

This study was financed in part by CAPES -- Brasil (Finance Code 001). The first author was supported in part by the Italian Ministry of the University and Research -- MUR, within the framework of the Call relating to the scrolling of the final rankings of the PRIN 2022 -- Project Code 2022HCLAZ8, CUP D53C24003370006 (PI A.~Palmieri, Local unit Sc.~Resp.~S.~Coriasco). 
The second and third author were supported in part by CNPq -- Brasil (grants 316850/2021-7 and 301573/2025-5, respectively).

\bibliographystyle{plain}
\bibliography{references}

\end{document}